\documentclass[12pt,letterpaper]{amsart}
\usepackage{euler, epic,eepic,latexsym, amssymb, amscd, amsfonts, xypic, url, epsfig, amsmath, colonequals}
\input xy
\xyoption{all}

\ProvideTextCommandDefault{\cprime}{\tprime}
 
\usepackage[backref]{hyperref}
\usepackage[alphabetic,backrefs,lite]{amsrefs}
\usepackage[usenames,dvipsnames]{color}

 \newlength{\baseunit}               
 \newcount{\numlines}                
\newtheorem{tm}{Theorem}
\newtheorem{pr}[tm]{Proposition}
\newtheorem{lm}[tm]{Lemma}
\newtheorem{co}[tm]{Corollary}
\newtheorem{df}[tm]{Definition}
\newtheorem{rmk}[tm]{Remark}

\newcommand{\GW}{\mathrm{GW}}
\newcommand{\W}{\mathrm{W}}

\newcommand{\rH}{\operatorname{H}}

\newcommand{\Z}{\mathbb{Z}}

\newcommand{\R}{\mathbb{R}}

\newcommand{\C}{\mathbb{C}}

\newcommand{\ZZ}{\Z}

\newcommand{\Hom}{\operatorname{Hom}} 
\newcommand{\End}{\operatorname{End}} 

\newcommand{\rank}{\operatorname{Rank}}

\newcommand\bbb[1]{\ensuremath{{\mathbf{#1}}}}

\newcommand{\Tr}{\operatorname{Tr}}
\newcommand{\ind}{\operatorname{ind}}

\newcommand{\MW}{\operatorname{MW}}

\newcommand{\PGL}{\operatorname{PGL}}
\newcommand{\GL}{\operatorname{GL}}

\newcommand{\Gr}{\operatorname{Gr}}

\newcommand{\Ker}{\operatorname{Ker}}

\newcommand{\sgn}{\operatorname{Sign}}

\DeclareMathOperator{\Res}{Res}
\newcommand{\ksep}{{k^{\operatorname{sep}}}}

\newcommand{\Proj}{\operatorname{Proj}}

\newcommand{\hidden}[1]{\footnote{Hidden:  #1}}
\renewcommand{\hidden}[1]{}

\newcommand{\bbP}{\mathbf{P}}

\newcommand{\calE}{{ \mathcal E}}

\newcommand{\calO}{{ \mathcal O}}

\newcommand{\calQ}{{ \mathcal Q}}

\newcommand{\calS}{{ \mathcal S}}

\newcommand{\calV}{{ \mathcal V}}

\newcommand{\Spec}{\operatorname{Spec}}
\newcommand{\Span}{\operatorname{Span}}
\newcommand{\ShHom}{\mathscr{H}\kern -.5pt om}

\newcommand{\Disc}{\operatorname{Disc}}

\DeclareMathOperator{\codim}{codim}
\makeatletter
\let\@wraptoccontribs\wraptoccontribs
\makeatother

\begin{document}
\pagestyle{plain}
\title{An arithmetic count of the lines meeting four lines in $\bbb{P}^3$}
\contrib[with an appendix by]{Borys Kadets, Padmavathi Srinivasan, Ashvin A.~Swaminathan, Libby Taylor, and Dennis Tseng}

\author{Padmavathi Srinivasan}
\address{Current: P.~Srinivasan, School of Mathematics, University of Georgia, 452 Boyd Graduate Studies, 1023 D. W. Brooks Drive, Athens, GA 30602.}
\email{Padmavathi.Srinivasan@uga.edu}
\urladdr{https://padmask.github.io/}

\author{Kirsten Wickelgren}
\address{Current: K.~Wickelgren, Department of Mathematics, Duke University, 120 Science Drive, Room 117 Physics, Box 90320, Durham, NC 27708-0320.}
\email{kirsten.wickelgren@duke.edu}
\urladdr{https://services.math.duke.edu/~kgw/}

\date{\today}
\begin{abstract}
We enrich the classical count that there are two complex lines meeting four lines in space to an equality of isomorphism classes of bilinear forms. For any field $k$, this enrichment counts the number of lines meeting four lines defined over $k$ in $\bbb{P}^3_k$, with such lines weighted by their fields of definition together with information about the cross-ratio of the intersection points and spanning planes. We generalize this example to an infinite family of such enrichments, obtained using an Euler number in $\bbb{A}^1$-homotopy theory. The classical counts are recovered by taking the rank of the bilinear forms.		
\end{abstract}
\maketitle


{\parskip=12pt 

\section{Introduction}

It is a classical result that there are exactly two lines meeting four general lines in $\bbb{P}^3_{\bbb{C}}$, and we briefly recall a proof. The lines meeting three of the four are pairwise disjoint and their union is a degree $2$ hypersurface. The intersection of the fourth line with the hypersurface is then two points, one on each of the two lines meeting all four. See \cite[3.4.1]{EisenbudHarris} for a more detailed description of this problem. If all four lines are now defined over an arbitrary field $k$, the associated degree $2$ hypersurface is also defined over $k$, but the two intersection points, and therefore the two lines, may have coefficients in some quadratic extension $k[\sqrt{L}]$ of $k$. For example, over the real numbers $\bbb{R}$, there may be two real lines or a Galois-conjugate pair of $\bbb{C}$-lines. 

In this paper, we give a restriction on the field of definition of the two lines combined with other arithmetic-geometric information on the configurations of the lines. More generally, we give an analogous restriction on the lines meeting $2n-2$ codimension $2$ hyperplanes in $\bbb{P}^n$ with $n$ odd. In the appendix, the condition that each of the hyperplanes is \textit{individually} defined over $k$ is relaxed to the condition that their \textit{union} is defined over $k$, or more precisely, that the lines form an \'etale $k$-algebra. In other words, the field of definition of any line is separable over $k$, and the Galois group of the maximal separable extension of $k$ is now allowed to permute the $2n-2$ hyperplanes.

These restrictions are equalities in the Grothendieck--Witt group $\GW(k)$ of the ground field $k$, defined to be the group completion of the semi-ring of isomorphism classes of non-degenerate, symmetric bilinear forms on finite dimensional vector spaces valued in $k$. The Grothendieck--Witt group arises in this context as the target of Morel's degree homomorphism in $\bbb{A}^1$-homotopy theory. A feature of $\bbb{A}^1$-homotopy theory is that it produces results over any field. These results can record arithmetic-geometric information about enumerative problems, classically posed over the complex numbers, which admit solutions using algebraic topology.  We show this is the case for the two enumerative problems just described and answer the question of what information is being recorded.

We now describe our results in more detail. Given a line $L$ meeting four pairwise non-intersecting lines $L_1$,$L_2$,$L_3$, and $L_4$, there are four distinguished points $L \cap L_1$, $L \cap L_2$, $L \cap L_3$ and $L \cap L_4$ on $L$. If $k(L)$ denotes the field of definition of $L$, then $L \cong \bbb{P}^1_{k(L)}$, and the four distinguished points on $L$ have an associated cross-ratio $\lambda_L \in k(L)$. The planes in $\bbb{P}^3_k$ containing $L$ are likewise parametrized by a scheme isomorphic to $\bbb{P}^1_{k(L)}$. This $\bbb{P}^1_{k(L)}$ also has four distinguished points on it, coming from the span of $L$ and each of the $L_i$, and hence an associated cross-ratio $\mu_L$.

Let $\langle a \rangle$ in $\GW(k)$ denote the stable isomorphism class of the one-dimensional bilinear form $k \times k \to k$ defined $(x,y) \mapsto a xy$, for $a$ in $k^*/(k^*)^2$. Since non-degenerate symmetric bilinear forms over fields may be stably diagonalized, it follows that $\langle a \rangle$ for $a \in k^*$ generate the group $\GW(k)$. For a separable field extension $k \subseteq E$, let $\Tr_{E/k}: \GW(E) \to \GW(k)$ denote the map which takes a bilinear form $\beta: V \times V \to E$ to the composition $\Tr_{E/k} \circ \beta: V \times V \to k$ of $\beta$ with the field trace $\Tr_{E/k}: E \to k$. Then the following equality holds in $\GW(k)$.

\begin{tm}\label{intro:4linesthm}
Let $k$ be a field of characteristic not $2$, and let $L_1$, $L_2$, $L_3$, and $L_4$ be general lines in $\bbb{P}^3_k$ defined over $k$. Then $$\sum_{\substack{\text{lines}~L~\text{such that } L \cap L_i \neq \emptyset \\ \text{for  } i=1,2,3,4}} \Tr_{k(L)/k} \langle \lambda_L - \mu_L \rangle = \langle 1 \rangle + \langle -1 \rangle.$$
\end{tm}

The condition that the lines are general means there is an open set of four-tuples of lines such that the theorem holds. In this case, this open set contains the lines that are pairwise non-intersecting and such that the fourth is not tangent to the quadric of lines meeting the first three. Since there are $2$ lines over the algebraic closure, the condition that the characteristic of $k$ is not $2$ guarantees that $k \subseteq k(L)$ is separable, allowing for the description of $\Tr_{k(L)/k} $ given above. 

Theorem~\ref{intro:4linesthm} generalizes as follows. Let $\pi_1,\pi_2, \ldots, \pi_{2n-2}$ be general codimension two hyperplanes defined over $k$ in $\bbb{P}^n_k$ for $n$ odd.  Suppose that $L$ is a line in $\bbb{P}^n$, defined over the field $k(L)$, which intersects all of the $\pi_i$. 

\begin{tm}\label{intro:2n-2codim2hyperplanesthm}
Let $k$ be a field. Let $n$ be an odd natural number and let $\pi_1,\pi_2, \ldots, \pi_{2n-2}$ be general codimension two hyperplanes defined over $k$ in $\bbb{P}^n_k$. Assume either that $k$ is perfect, or that the extension $k \subseteq k(L)$ is separable for every line $L$ that meets all the planes $\pi_i$. Let $i(L) \in k(L)^*/(k(L)^*)^2$ be the index of the line $L$ defined by Equation~\ref{i(L)def}. (It records information about $k(L)$, the intersection points $\pi_i \cap L$, and the codimension $1$ hyperplanes spanned by $\pi_i$ and $L$.) Then $$\sum_{\substack{\text{lines}~L~\text{such that } L \cap \pi_i \neq \emptyset \\ \text{for  } i=1\ldots 2n-2}} \Tr_{k(L)/k} \langle i(L) \rangle =  \frac{1}{2} \frac{(2n-2)!}{n! (n-1)!}(\langle 1 \rangle + \langle -1 \rangle).$$
\end{tm}

In the case where $n=3$, this theorem reduces to Theorem \ref{intro:4linesthm} with the difference that the hypothesis that the characteristic of $k$ is not $2$ is relaxed to only require that $k \subseteq k(L)$ be separable.

There are many tools available for studying $\GW(k)$. These can be used in conjunction with enumerative results valued in the Grothendieck--Witt group such as Theorems \ref{intro:4linesthm} and \ref{intro:2n-2codim2hyperplanesthm} and we briefly indicate a few of these tools. For simplicity, assume the characteristic of $k$ is not $2$. The fundamental ideal of $\GW(k)$ is the kernel of the rank homomorphism $\rank: \GW(k) \to \ZZ$ taking the isomorphism class of a bilinear form to its rank. It's powers define a filtration of $\GW(k)$. The Milnor conjecture, proven by Voevodsky and Orlov--Vishik--Voevodsky using breakthrough work of the former, identifies the associated graded of this filtration with the \'etale cohomology group $\rH^*(k, \Z/2)$ and the mod $2$ Milnor $K$-theory. These isomorphisms gives rise to invariants valued in $\rH^*(k, \Z/2)$, or equivalently Milnor K-theory, the first of which are the rank, discriminant, Hasse-Witt, and Arason invariants \cite{Milnor_AlgK-theory_quadratic_forms} \cite{OrlovVishikVoevodskay} \cite{Voevodsky-Motivic_cohomology_Z2-coeffs} \cite{Voevodsky-Reduced_power_ops_motivic_cohomology}. Analogous results in characteristic $2$ are proven in \cite{KatoBilforms}. For many fields, much is understood about $\GW(k)$, for example giving algorithms to determine if two given sums of the generators $\langle a \rangle$ for $a \in k^*/(k^*)^2$ of $\GW(k)$ are equal, as well as computations of $\GW(k)$. See for example \cite{lam05}. Applying invariants of $\GW(k)$ to the equalities of Theorems \ref{intro:4linesthm} and \ref{intro:2n-2codim2hyperplanesthm} produces other equalities, which may be valued in more familiar groups. A selection of such results follows. 

Let the field $k$ be  the real numbers,  $k= \bbb{R}$. Applying the signature to both sides of Theorem \ref{intro:4linesthm}, we see that if the two lines are real, the sign of $ \lambda_L - \mu_L$ must be reversed for the two lines. If the two lines are a complex conjugate pair, the theorem gives no information. More generally, in the situation of Theorem \ref{intro:2n-2codim2hyperplanesthm}, half of the real lines will have $i(L)$ negative and half positive. 

\begin{rmk}
The question of which lines are real in both these situations is the subject of extensive work. Sottile has shown that the lines may all be real \cite[Theorem C]{Sottile-enumerative_geometry_for_real_Grassmannian_lines_Pn}. There are also connections to the B. and M. Shapiro conjecture, proven by Eremenko and Gabrielov \cite{EremenkoGabrielov-RationalFunctionsRealCriticalPoints} and Mukhin, Tarasov, and Varchenko \cite{MukhinTarasovVarchenko-ShapiroConjecture} \cite{MukhinTarasovVarchenko-SchubertCalculusGL}. For example, the conjecture's positive solution gives a large class of examples where the lines meeting four lines in space are real \cite[1.2.1]{Sottile-Real_solutions_to_equations_from_geometry}. More generally, see \cite[ Chapters 8,9, and 10]{Sottile-Real_solutions_to_equations_from_geometry} for a lovely exposition of the Shapiro conjecture. Vakil's Schubert induction \cite{Vakil_Schubert_induction} shows that all Schubert problems are enumerative over $\bbb{R}$ in the sense that all the solutions are real. See \cite[Proposition 2.4]{Vakil_Schubert_induction}. His work on the Galois groups of enumerative problems over any base moreover gives information about the fields of definition of the solutions \cite[Section 2.9]{Vakil_Schubert_induction}. 
\end{rmk}

For $k = \bbb{F}_q$ a finite field with $q$ elements and $q$ odd, $\GW(k) \cong \Z \oplus \Z/2\Z$ given by the rank and the discriminant maps. Applying the discriminant to Theorem~\ref{intro:4linesthm} produces, for example:

\begin{co}\label{Fqcorthm4lines}
Let $k = \bbb{F}_q$ be a finite field with $q$ elements, with $q$ odd. Let $L_1$,$L_2$,$L_3$,$L_4$ be general lines defined over $k$ in $\bbb{P}^3_k$. If a line $L$ meeting $L_i$ for $i=1,\ldots,4$ is defined over $\bbb{F}_{q^2}$, then $$\lambda_L - \mu_L = \begin{cases}   &\mbox{is a non-square for } q \cong 1 \mod 4 \\ 
&\mbox{is a square for } q \cong 3 \mod 4 .\end{cases}$$
\end{co}

The proofs of Theorems \ref{intro:4linesthm} and \ref{intro:2n-2codim2hyperplanesthm} and Corollary \ref{Fqcorthm4lines} can be found in Section~\ref{mainthm_Section}. For these proofs, we use a strategy based on joint work \cite{CubicSurface} of Jesse Kass and the second named author. Namely, we take a classical enumerative problem over the complex numbers that admits a solution using the Euler number from algebraic topology, and rework it over a field $k$ using an enriched Euler number valued in $\GW(k).$ This enriched Euler number was defined in \cite[Definition 33]{CubicSurface} to be a sum of certain local indices. To obtain an enumerative result over $k$, one then needs a geometric interpretation of these local indices, which is guessed on a case-by-case basis. 

In the present case, there is a classical count of the appropriate number of complex lines as a power of the first Chern class of the line bundle $\calS^* \wedge \calS^*$ on an appropriate Grassmannian, where $\calS^* \wedge \calS^*$ denotes the wedge of the dual tautological bundle with itself. This characteristic number is equivalent to the Euler number of $\oplus_{i=1}^N \calS^* \wedge \calS^*$ for an appropriate $N$. The Euler number constructed in \cite{CubicSurface} is a class in $\GW(k)$ associated to a relatively oriented vector bundle of rank $r$ on a smooth, proper $r$-dimensional scheme over $k$, which admits sections with only isolated zeros. Picking one such section, the Euler number is then the sum over these isolated zeros of a local index associated to each zero, which can be expressed as a local degree in $\bbb{A}^1$-homotopy theory. As in the classical case, a configuration of codimension $2$ hyperplanes (or more precisely the set of equations whose zero loci are the hyperplanes) determines a section of $\oplus_{i=1}^N \calS^* \wedge \calS^*$. We therefore have that a fixed element of $\GW(k)$, the Euler number, is a sum over the lines of the local degree in $\GW(k)$ of a section. This equality generalizes the fixed $\bbb{Z}$-valued Euler number on the complex Grassmiannian from classical algebraic topology expressed as a sum over the lines of the local $\bbb{Z}$-valued degrees at the zero locus of a section. In the complex case, these latter local degrees happen to be generically all one because complex manifolds and algebraic sections are orientable, giving the number of lines as the Euler number. Over other fields, interesting local degrees or indices arise. 

Readers who would like to avoid $\bbb{A}^1$-homotopy theory may do so, as the construction of the Euler number of \cite{CubicSurface} uses the classes from the Eisenbud--Khimshiashvili--Levine Signature Formula as local indices. By \cite{KWA1degree} and \cite{BBMMOTraceA1degree}, these local indices are the local $\bbb{A}^1$-degrees, but they also have a concrete construction from commutative algebra that can be programmed into a computer. Note, however, that it is more than just analogy that links our results to algebraic topology; there is a full-fledged theory of $\bbb{A}^1$-homotopy theory providing a connection \cite{morelvoevodsky1998} \cite{morel}. Moreover, there is a machine producing enriched results of the form given in Theorems \ref{intro:4linesthm} and \ref{intro:2n-2codim2hyperplanesthm}. Once the machine has produced an enrichment, there is guess-work involved in identifying local indices, but once this is accomplished, the end result is by design independent of $\bbb{A}^1$-homotopy theory, and one can then, at least in some cases, provide alternate proofs that are also independent. 

A main tool here is the Euler number of \cite{CubicSurface}. There are older constructions of Euler classes in $\bbb{A}^1$-homotopy theory in \cite{BargeMorel} \cite[8.2]{morel}, more recent ones in \cite{DJK} \cite{LevineRaksit_MotivicGaussBonnet}, and other constructions of Euler classes or numbers for schemes independent of $\bbb{A}^1$-homotopy theory in \cite{Grig_Ivan} \cite{Mandal-Srid-Euler_classes_complete_intersections} \cite{BhatSrid-Zero_cycles} \cite{BhatSrid-Euler_class_group} \cite{bhatwadekar06}. The latter four are of a different focus. The construction in \cite{Grig_Ivan} is the image of that of \cite{CubicSurface} after quotienting by trace forms, and is discussed in more detail in \cite[1.1]{CubicSurface}. Compatibility of those of \cite{BargeMorel} \cite{DJK} \cite{CubicSurface} \cite{LevineRaksit_MotivicGaussBonnet} \cite[8.2]{morel} is shown in \cite{BW-Euler}. See also \cite{FaselGroupesCW} \cite{AsokFasel_comp_euler_classes} \cite{Levine-EC} for results on Euler classes and useful tools for their computation.

Matthias Wendt has a lovely alternate computation of the Euler classes of $\oplus_{i=1}^N \calS^* \wedge \calS^*$, using a Schubert Calculus he has developed \cite{Wendt-oriented_schubert}. He also considers the resulting applications to enumerative geometry. His results as well as methods are different from the ones given here. His work \cite{Wendt-oriented_schubert} builds on his previous work \cite{Wendt-Chow-Witt_Gr} and his joint work \cite{Wendt-Hornbostel} with Hornbostel. 

Unique to the present paper are the given computations of the local $\GW(k)$-degrees or indices of the lines and their geometric interpretations, the resulting Theorems \ref{intro:4linesthm} and \ref{intro:2n-2codim2hyperplanesthm}, and consequences. In the case of Theorem \ref{intro:4linesthm}, we also find the stronger result that the cross ratios of the points and the planes switch when we switch the two lines over their fields of definition, i.e., $\mu_L = \lambda_{\tilde{L}}$ and $\mu_{\tilde{L}} = \lambda_{L}$ (see Theorem \ref{thmcrossratio} and Example \ref{example}). This can be verified independently. We give computations of the relevant Euler numbers and take the opportunity to further the study of the Euler number of \cite{CubicSurface}. 

This paper fits into a recent program that could be called $\bbb{A}^1$-enumerative geometry, or enumerative geometry enriched in quadratic forms. See also \cite{BW-Euler} \cite{Hoyois_lef} \cite{KWA1degree} \cite{KWA1degreeClassical} \cite{CubicSurface} \cite{Levine-EC} \cite{Levine-NormalCone}  \cite{Levine-Witt} \cite{Levine-Welschinger} \cite{McKean-Bezout} \cite{Pauli-quintic3fold} \cite{Wendt-oriented_schubert}.

\subsection{Outline}

In Section~\ref{setup}, we give the necessary results and notation to have a well-defined Euler number of $\oplus_{i=1}^N \calS^* \wedge \calS^*$ in $\GW(k)$ as a sum over lines of a local index or degree. In Section~\ref{LinearAlgebraFormula}, we give formulas for the local index, in particular in terms of the $i(L)$ for Theorem~\ref{intro:2n-2codim2hyperplanesthm} above. In Section~\ref{EulercharV}, we give computations of Euler numbers, one using arguments of Fasel and Levine. In Section~\ref{crossratio}, we prove the connection between the local indices and the cross-ratios appearing in Theorem~\ref{intro:4linesthm}. Section~\ref{mainthm_Section} contains the proofs of the stated results in the introduction and an explicit example. 

\section{Local coordinates on Grassmannians, orientations on vector bundles, and Euler numbers}\label{setup}

\subsection{Preliminaries on relative orientations on vector bundles and compatible trivializations}\label{subsection:prelim_vb_or}
Let $k$ be a field, and let $X$ be a smooth $k$-scheme.
A vector bundle $V \to X$ is {\em oriented} (sometimes called {\em weakly oriented}) if it is equipped with a choice of isomorphism, called an {\em orientation}, $\det V \cong \mathcal{L}^{\otimes 2}$ for a line bundle $\mathcal{L}$ on $X$. A smooth scheme $X$ is oriented if its tangent bundle is. Let $U$ be a Zariski open subset of $X$. A section in $\mathcal{L}^{\otimes 2}(U)$ is called a {\em square} if it is of the form $s \otimes s$ for $s$ in $\mathcal{L}(U)$. A trivialization $\psi: V\vert_U \stackrel{\cong}{\to} \mathcal{O}\vert_U^r$ is {\em compatible} with a given orientation if the composition $\det \mathcal{O}^r\vert_U \to \det V\vert_U \to \mathcal{L}^{\otimes 2}$ takes the canonical section in $\det \mathcal{O}^r(U)$ to a square. $V$ is {\em relatively oriented} when $\Hom(\det TX, \det V)$ is oriented. While we do not need the notion of relative orientation for this paper, we prove some results in this greater generality for their own interest.

\subsection{Euler numbers for relatively oriented bundles}\label{subsection:Euler_numbers_relatively_or_bundles}
We use the Euler number of \cite{CubicSurface}. Namely, let $\calE$ be a relatively oriented vector bundle of rank $r$ on a smooth proper $k$-scheme $X$ of dimension $n=r$, equipped with a section $\sigma$ with only isolated zeros. For each zero $p$ of $\sigma$, there is an associated local index $\ind_p \sigma \in \GW(k)$ \cite[Definition 28, Corollary 29]{CubicSurface}, which we will discuss in some more detail in Section \ref{explicitformula} for the convenience of the reader. This local index is an $\bbb{A}^1$-degree in the sense of Morel \cite{morel} \cite{KWA1degree} \cite{BBMMOTraceA1degree}. The Euler number $e(\calE, \sigma)$ is defined to be the sum over the zeros of $\sigma$ of the local indices:
$$e(\calE, \sigma) = \sum_{p: \sigma(p) = 0} \ind_p \sigma.$$       

The results of the following lemma were part of the motivation for the definition of $e(\calE, \sigma)$. We provide a formal proof for completeness. Suppose $k$ is a subfield of $\C$. Then there is an associated complex manifold $X(\C)$ and complex vector bundle $\calE(\C)$, which has an Euler class $e_{\C}$. Similarly, if $k$ is a subfield of $\R$, there is an associated topological Euler number $e_{\R}$ of the relatively oriented bundle $\calE(\R) \to X(\R)$. Let $\rank$ and $\sgn$ denote the rank and signature, respectively, of a bilinear form or element of $\GW(\C)$ or $\GW(\R)$ respectively.

\begin{lm}\label{lm:rk_sgn_e}
 Let $k$, $\calE$, and $X$ be as above.
\begin{enumerate}
\item \label{it:rank=eC} If $k \subset \C$, then $\rank e(\calE, \sigma) = e_{\C}$.
\item \label{it:sgn=eR} If $k \subset \R$, then $\sgn e(\calE, \sigma) = e_{\R}$.
\end{enumerate}
\end{lm}

\begin{proof}
The topological Euler number of $\calE(\C)$ or $\calE(\R)$ can be computed as a sum of over the zeros of $\sigma$ of a local index, $$ e_{\C} =  \sum_{p: \sigma(\C)(p) = 0} \ind_p \sigma(\C) $$ $$ e_{\R} =  \sum_{p: \sigma(\R)(p) = 0} \ind_p \sigma(\R),$$ where $\sigma(\C)$ (respectively $\sigma(\R)$) denotes the associated smooth section on complex (respectively real) points. See \cite[Theorem 11.16]{BottTu} \cite[Section 1]{okonek14}. Choose a hermitian metric (respectively metric) on $\calE(\C)$ (respectively $\calE(\R)$) and let $E$ denote the corresponding sphere bundle. The index $\ind_p \sigma(\C)$ (respectively $\ind_p \sigma(\R)$) is defined to be the topological degree of the composite map \begin{equation}\label{map_Cspheres_aroundp}\partial  \overline{D_r(\C)} \stackrel{\sigma(\C)}{\longrightarrow} E\vert_{\overline{D_r(\C)}} \cong  \overline{D_r(\C)} \times S^{2r-1} \to S^{2r-1} \end{equation} \begin{equation}\label{map_Rspheres_aroundp}(\text{respectively }\partial  \overline{D_r(\R)} \stackrel{\sigma(\R)}{\longrightarrow} E\vert_{\overline{D_r(\R)}} \cong  \overline{D_r(\R)} \times S^{r-1} \to S^{r-1} )\end{equation} where $\overline{D_r}$ denotes the closure of a small disk around $p$ and the final maps in the given compositions are projections. By \cite[Corollary 4]{Palamodov}, the degree of \eqref{map_Cspheres_aroundp} is the rank of $\ind_p \sigma$, proving \eqref{it:rank=eC}. By the Eisenbud--Khimshiashvili--Levine Signature Theorem \cite{khimshiashvili} \cite{eisenbud77}, the degree of \eqref{map_Rspheres_aroundp} is the signature of $\ind_p \sigma$, proving \eqref{it:sgn=eR}.
\end{proof}

\begin{rmk}
A stronger version of this lemma in the case $k \subset \R$ follows from work of Hornbostel--Wendt--Xie--Zibrowius \cite{HWXZ-real_cycle}: it is not necessary to assume that there is a section $\sigma$ with only isolated zeros, as can be seen by combining \cite[Proposition 6.1 and Section 4.3]{HWXZ-real_cycle} and using \cite[Corollary p 3]{BW-Euler} to identify $e(\calE, \sigma)$ with the pushforward of the algebraic Euler class of \cite[Proposition 6.1]{HWXZ-real_cycle}.
\end{rmk}

\subsubsection{A formula for the local degree using compatible local coordinates}\label{explicitformula}
We now recall a computation of the local degree or index $\ind_p \sigma \in \GW(k)$ from \cite[Section4]{CubicSurface}.

Choose local coordinates on $X$ and a local trivialization of $\calE$ which are compatible with the relative orientation of $\calE$. See \cite[Definition 19]{CubicSurface} for the general definition of such coordinates and trivializations. In our cases, $TX$ and $\calE$ will both be oriented and the choice of local coordinates amounts to finding open neighborhoods of each point isomorphic to affine spaces $\bbb{A}^r_k$ such that the induced trivialization of $TX$ is compatible with the orientation in the sense of Section \ref{subsection:prelim_vb_or}. This identifies $\sigma$ with a function $\sigma: \bbb{A}_k^r \to \bbb{A}_k^r$. If $\sigma$ has a zero at a point $p$ such that the corresponding function $(f_1, \ldots, f_r): \bbb{A}_k^r \to \bbb{A}_k^r$ satisfies the condition that the Jacobian determinant $J=\det (\frac{\partial f_i}{\partial x_j})_{ij}$ is non-zero in $k(p)$, then we say that $\sigma$ has a {\em simple} zero. The local degree at a simple zero $p$ such that $k \subseteq k(p)$ is a separable field extension is computed in \cite[Proposition 32]{CubicSurface}, and we include the statement for clarity. 

\begin{pr}\label{CubicSurfacetrJprop}
Suppose $\sigma$ has a simple zero at a point $p$. Then the corresponding local degree $\ind_p \sigma$  is $\Tr_{k(p)/k} \langle J \rangle$ if $k \subseteq k(p)$ is a separable field extension.
\end{pr}

\subsection{Vector bundles on Grassmannians}
We will now introduce the relevant vector bundles for our enumerative problems and equip them with relative orientations. We give local trivializations and coordinates compatible with the relative orientations.

Let $n$ be an odd integer, and let $\Gr(2,n+1)$ denote the Grassmannian parametrizing $2$-dimensional {\em{subspaces}} of $k^{n+1}$, or equivalently, lines in $\bbb{P}^n_k$. For any field extension $k \subseteq E$ and a $k$-scheme $X$, we will denote the base change of $X$ to $E$ by $X_E$, and similarly for vector bundles over $X$. By a {\em line} $L$ in $\bbb{P}^n_k$, we mean a closed point of this Grassmannian. We write $k(L)$ for the residue field of this closed point and call this the field of definition of $L$. After basechange to $k(L)$, the line $L$ corresponds to a closed subscheme of $\bbb{P}^n_{k(L)}$.

Let $\calS$ be the tautological bundle on $\Gr(2,n+1)$. Then the line bundle $\calO(1)$ corresponding to the Plucker embedding of $\Gr(2,n+1)$ is $\Lambda^2 \calS^*$. Let $\calQ$ denote the quotient bundle on $\Gr(2,n+1),$ defined as the cokernel of the inclusion of the tautological bundle into the trivial bundle $$0 \to \calS \to \bbb{A}^{n+1} \to \calQ \to 0. $$

\subsubsection{Compatible local trivialization for the tangent bundle}

Let $\{e_1,\ldots,e_{n+1}\}$ denote a basis for $k^{n+1}$, and let $\{\phi_1,\ldots,\phi_{n+1}\}$ denote the dual basis of $(k^{n+1})^*$. The $2$-dimensional subspace $k e_{n} \oplus k e_{n+1}$ spanned by $\{e_{n}, e_{n+1} \}$ determines a $k$-point of $\Gr(2,n+1)$. There are local coordinates $$\Spec k[x_1,\ldots,x_{n-1}, y_1,\ldots, y_{n-1}] \to \Gr(2,n+1)$$ of $\Gr(2,n+1)$ around this point such that $(x_1,\ldots,x_{n-1}, y_1,\ldots, y_{n-1})$ corresponds to the span of  $\{\tilde{e}_{n}, \tilde{e}_{n+1} \}$, where $\{\tilde{e}_{1}, \ldots, \tilde{e}_{n+1} \}$ is the basis of $k^{n+1}$ defined by $$  \begin{cases} e_i  &\mbox{for } i = 1,\ldots, n-1 \\ 
\sum_{i=1}^{n-1} x_i e_i + e_{n} & \mbox{for } i=n
\\ 
\sum_{i=1}^{n-1} y_i e_i + e_{n+1} & \mbox{for } i={n+1}. \end{cases} $$ These local coordinates determine a local trivialization of $T\Gr(2,n+1)$ using the canonical trivialization of the tangent space of $\bbb{A}^{2(n-1)}$. We will say that coordinates are {\em compatible} with a given orientation if the corresponding local trivialization of the tangent bundle is. 
Let $\{\tilde{\phi}_{1}, \ldots, \tilde{\phi}_{n+1} \}$ be the dual basis for the basis $\{\tilde{e}_{1}, \ldots, \tilde{e}_{n+1} \}$ of $k^{n+1}$. The content of the following proof is contained in \cite[Lemma 42]{KWA1degree}, but we include the proof in the stated generality for completeness.

\begin{lm}\label{lm:label_oriented_coordinates_Gr}
Let $n$ be odd. There is an orientation of $T\Gr(2,n+1)$ such that the local coordinates given by the maps $\Spec k[x_1,\ldots,x_{n-1}, y_1,\ldots, y_{n-1}] \to \Gr(2,n+1)$ just described are compatible.
\end{lm}

\begin{proof}
Let $\{e_1,\ldots,e_{n+1}\}$ and $\{e_1',\ldots,e_{n+1}'\}$ denote two chosen bases for $k^{n+1}$.  As above, consider the corresponding local coordinates on the open subsets $U$ and $U'$ of $\Gr(2,n+1)$. Under the canonical identification of $T \Gr(2,n+1)$ with $\Hom(\calS, \calQ)$ the corresponding trivialization $T \Gr(2,n+1)\vert_U \cong \calO_U^{2(n-1)}$ corresponds to the basis of $T \Gr(2,n+1)(U)$ given by $$\{ \tilde{\phi}_{n} \otimes \tilde{e}_i : i =1, n-1\} \cup \{ \tilde{\phi}_{n+1} \otimes \tilde{e}_i : i =1, n-1\}$$ and similarly for $U'$. (Note the slight abuse of notation when we consider, say $\tilde{e}_1$ at the point $p$ to be an element of $k(p)^{n+1}/(k(p) \tilde{e}_n \oplus k(p) \tilde{e}_{n+1})$.) These trivializations determine clutching functions for $T\Gr(2,n+1)$, i.e., isomorphisms \begin{equation}\label{clutching-change-basis}\calO^{2(n-1)}\vert_{U \cap U'} \to T\Gr(2,n+1)\vert_{U \cap U'} \to \calO^{2(n-1)}\vert_{U \cap U'},\end{equation} which are given by the change-of-basis matrix $M$ in $\GL_{2(n-1)} (U \cap U')$
relating $$\tilde{\phi}_{n} \otimes \tilde{e}_1, \tilde{\phi}_{n} \otimes \tilde{e}_2, \ldots, \tilde{\phi}_{n} \otimes \tilde{e}_{n-1},  \tilde{\phi}_{n+1} \otimes \tilde{e}_1, \tilde{\phi}_{n+1} \otimes \tilde{e}_2, \ldots, \tilde{\phi}_{n+1} \otimes \tilde{e}_{n-1}$$ to $$\tilde{\phi}'_{n} \otimes \tilde{e}'_1, \tilde{\phi}'_{n} \otimes \tilde{e}'_2, \ldots, \tilde{\phi}'_{n} \otimes \tilde{e}'_{n-1},  \tilde{\phi}'_{n+1} \otimes \tilde{e}'_1, \tilde{\phi}'_{n+1} \otimes \tilde{e}'_2, \ldots, \tilde{\phi}'_{n+1} \otimes \tilde{e}'_{n-1}.$$ 

The expressions $\tilde{\phi}_i(e_j')$ and $\tilde{\phi}'_i(e_j)$ determine regular functions on $U \cap U'$, because $\tilde{\phi}_i$ and $\tilde{\phi}_i'$ are sections of $\calS^*(U \cap U')$ and $e_j$ and $e_j'$ are sections of $\calS(U \cap U')$. At any point $p$ of $U \cap U'$, the span of $\{  \tilde{e}_n, \tilde{e}_{n+1}\}$ and $\{  \tilde{e}_n', \tilde{e}_{n+1}'\}$ are equal, i.e., $k(p) \tilde{e}_n \oplus k(p) \tilde{e}_{n+1} = k(p) \tilde{e}_n' \oplus k(p) \tilde{e}_{n+1}'$, and both give bases of the fiber of $\calS$ at $p$. In the same way, $\{  \tilde{\phi}_n, \tilde{\phi}_{n+1}\}$ and $\{  \tilde{\phi}_n', \tilde{\phi}_{n+1}'\}$ are bases for $\calS^*$ over $U \cap U'$ and the change of basis matrix is given by $$A = \begin{bmatrix}
\tilde{\phi}'_n (\tilde{e}_n)& \tilde{\phi}'_n (\tilde{e}_{n+1})\\
\tilde{\phi}'_{n+1} (\tilde{e}_n)& \tilde{\phi}'_{n+1} (\tilde{e}_{n+1})
\end{bmatrix}   $$ Thus $A$ is a $2\times 2$ matrix $A \in \GL_2(U \cap U')$ with $\det A$ in $\calO^*(U \cap U')$. Similarly, the images of $\{\tilde{e}_1, \ldots, \tilde{e}_{n-1} \}$ and $\{\tilde{e}'_1, \ldots, \tilde{e}'_{n-1} \}$ are bases for $\calQ$ over $U \cap U'$ the change of basis matrix for and the change of basis matrix is an $(n-1) \times (n-1)$ matrix $B\in \GL_{n-1}(U \cap U')$, and in particular $\det B$ in $\calO^*(U \cap U')$. Then $M$ is the tensor product $M = A \otimes B$. It follows that $\det M = \det A^{n-1} \det B^2$. Since $n$ is assumed to be odd, we have that $\det M$ is the square of $\det A^{(n-1)/2} \det B$. Since $A$ and $B$ are the clutching functions of vector bundles, $\det A^{(n-1)/2} \det B$ satisfies the cocycle condition to define the clutching functions of a line bundle defining an orientation of $T\Gr(2,n+1)$ compatible with the described local coordinates.

\end{proof}

\subsubsection{The vector bundles for our enumerative problems}\label{vecandsec}

\begin{df}
 Let $\calV$ be the rank $2n-2$ vector bundle $\oplus_{i=1}^{2n-2} \Lambda^2 \calS^*$. 
\end{df}

Let $\pi_1,\pi_2,\ldots,\pi_{2n-2}$ be $2n-2$ codimension two subspaces of $k^{n+1}$. For each $\pi_i$, choose a basis for the annihilator of $\pi_i$ in $(k^{n+1})^*$, namely, two linearly independent linear forms $\alpha_i$ and $\beta_i$ in $(k^{n+1})^*$ vanishing on $\pi_i$. Let $\sigma$ be the section of $\calV$ given by $\sigma = (\alpha_1 \wedge \beta_1,\ldots,\alpha_{2n-2} \wedge \beta_{2n-2})$. (More explicitly, any point $L$ of $\Gr(2,n+1)$ corresponds to a $2$-dimensional subspace $W$ of $k(L)^{\oplus n+1}$. The elements $\alpha_i$ and $\beta_i$ tensored with $k(L)$ then restrict to functionals $\alpha_i\vert_{W}$ and $\beta_i\vert_{W}$ in $W^*$. The fiber of $\Lambda^2 \calS^*$ at $L$ is canonically identified with $W^* \wedge W^*$, and $ (\alpha_1\vert_W \wedge \beta_1\vert_W,\ldots,\alpha_{2n-2}\vert_W \wedge \beta_{2n-2}\vert_W)$ determines the section $\sigma$.)

The following lemma is standard, but we include it for clarity.

\begin{lm}
Let $L$ be a point of  $\Gr(2,n+1)$ and $\sigma$ the section of $\calV$ as described above. Then $\sigma(L)=0$ if and only if $L$ meets all of the hyperplanes $\pi_1,\pi_2,\ldots,\pi_{2n-2}$. 
\end{lm}
\begin{proof}
If a codimension $2$ hyperplane $\pi$ is cut out by the two linear forms $\alpha,\beta$ and a dimension $2$ subspace $L$ is spanned by the vectors $e,f$, then $(\alpha \wedge \beta)(e \wedge f) = 0$ if and only if both $\alpha$ and $\beta$ simultaneously vanish on some linear combination of $e$ and $f$, or in other words, if and only if $\pi \cap L \neq \emptyset$. 
\end{proof}

\begin{lm}\label{Voriented}
 There is an orientation of the vector bundle $\calV \colonequals \oplus_{i=1}^{2n-2} \Lambda^2 \calS^*$ on $\Gr(2,n+1)$ determined by the local trivialization of $\calV$ coming from any local trivialization of $\calO(1)$.
\end{lm}
\begin{proof}
The top exterior power $\Lambda^{2n-2} \calV \cong ((\Lambda^2 \calS^*)^{\otimes n-1})^{\otimes 2} \cong \calO(2n-2)$ inherits a local trivialization whose transition functions are the $(2n-2)$th power of the transition functions for $\calO(1)$ using the same cover. Thus these local trivializations determine an isomorphism $\Lambda^{2n-2} \calV \cong (\calO(1)^{\otimes (n-1)})^{\otimes 2}$. Changing the local trivialization of $\calO(1)$ changes this isomorphism by a square, which gives an isomorphic orientation.
\end{proof}

\subsubsection{Identifying sections of $\calV$ with isolated zeroes}

We now identify many sections of $\calV$ with isolated zeros so that we may use the explicit formula for the Euler number of Section~\ref{subsection:Euler_numbers_relatively_or_bundles}.

Fix oriented local coordinates on $\Gr(2,n+1)$ and the orientation on $\calV$ and compatible local trivialization of $\calV$ described in Lemma~\ref{Voriented}. Let $i \colon G \colonequals \Gr(2,n+1) \rightarrow \bbP^N$ be the Pl{\"{u}}cker embedding. Let $\check{\bbb{P}^N}$ be the dual projective space parametrizing hyperplanes in $\bbb{P}^N$. Let
 \[ Z \colonequals \{ (H_1,H_2,\ldots,H_{2n-2}) \in (\check{\bbb{P}^N})^{2n-2} \ | \ \dim (G \cap H_1 \cap \cdots \cap H_{2n-2}) \neq 0 \} .\]
Let $S \colonequals \{ (\sigma_1,\ldots,\sigma_{2n-2}) \in H^0(G, \Lambda^2 \calS^*)^{2n-2} \ | \ \sigma_i \neq 0 \ \forall \ i \}$. We have a natural map $$p \colon S \rightarrow (\check{\bbb{P}^N})^{2n-2},$$ which is $\bbb{G}_{m,k}^{2n-2}$-bundle, and we let $U \colonequals S \setminus p^{-1}(Z)$. The purpose of the following lemma and corollary is to show that the complement of $U$ is codimension at least $2$, from which it will follow that $e(\calV, \sigma)$ is independent of $\sigma$ for $\sigma$ in $U$.

\begin{lm}\label{Bezout}
 The set $Z$ is closed and $\codim Z \geq 2$.
\end{lm}
\begin{proof}
 For any irreducible subvariety $W \subset \bbb{P}^N$ with $\dim W \geq 1$ and for any hyperplane $H \in \check{\bbb{P}^N}$, we have $\dim (W \cap H) = \dim W-1$ if and only if $H \not\supset W$, and $\dim W = \dim (W \cap H)$ otherwise (see, e.g., \cite[I 6.2 Theorem 5]{Shafarevich_Basic_algebraic_geometry1}). By applying this in turn to the finitely many irreducible components of each of $W=G,G\cap H_1,\ldots, G\cap H_1 \cap H_2 \cap \cdots \cap H_{2n-3}$, we see that $Z = \bigcup_{i=1}^{2n-2} \pi_i^{-1} (Z_i)$, where $Z_i \subset (\check{\bbb{P}^N})^{i}$ is the subset  \begin{align*} Z_i \colonequals \{ (H_1,H_2,&\ldots,H_{i}) \in (\check{\bbb{P}^N})^{i} \ |\\ &\ H_i \supset A, \textup{ for some irreducible component } A \textup{ of } G \cap H_1 \cap \cdots \cap H_{i-1}  \}, \end{align*} and $\pi_i: (\check{\bbb{P}^N})^{2n-2} \to (\check{\bbb{P}^N})^i$ is the projection on the first $i$ factors. Note that $Z_i$ could equivalently be written $$ Z_i =   \{ (H_1,H_2,\ldots,H_{i}) \in (\check{\bbb{P}^N})^{i} \
 | \dim (G \cap H_1 \cap \cdots \cap H_{i}) > 2n-2 - i \}.$$ 

We will first show $Z_i$ is a closed subset of $(\check{\bbb{P}^N})^{i}$. To see this, consider the incidence variety $I \subset (\check{\bbb{P}^N})^{i} \times \bbb{P}^N$ defined $$I \colonequals \{ (H_1,H_2,\ldots,H_{i}, x): H_1,\ldots,H_i \in \check{\bbb{P}^N}, x \in G \cap H_1 \cap \cdots \cap H_i \}.$$ The projection map restricted to $I$ gives a map $f: I \to (\check{\bbb{P}^N})^{i}$, and $Z_i$ is the locus in  $(\check{\bbb{P}^N})^{i}$ of points where the fiber has larger than expected dimension, which is closed by, e.g., \cite[I 6.3 Theorem 7]{Shafarevich_Basic_algebraic_geometry1}.

We will now show $\codim Z_i \geq 2$. Consider the map $\pi^0_i:  (\check{\bbb{P}^N})^{i} \to (\check{\bbb{P}^N})^{i-1}$ given by projection onto the first $i-1$ factors. For each $(H_1, \ldots, H_{i-1})$ in $(\check{\bbb{P}^N})^{i-1}$ and each irreducible component $A$ of $G \cap H_1 \cap \cdots \cap H_{i-1}$, the dimension of $A$ satisfies $\dim (A) \geq 2n-2-(i-1) \geq 1$, and therefore also $\dim A \geq 1$. Let $P,Q$ be two distinct points on $A$. Then $Z'_A \colonequals \{ H \in \check{\bbb{P}^N} \ | \ H \supset  A \} \subset \check{\bbb{P}^N}$ is contained in the intersection of the two distinct hyperplanes $\{ H \in \check{\bbb{P}^N} \ | \ P \in H \}$ and $\{ H \in \check{\bbb{P}^N} \ | \ Q \in H \},$ and therefore has codimension $\geq 2$. Thus the fibers of the restriction of $\pi_i^0$ to $Z_i$ have codimension $\geq 2$. It follows (\cite[I 6.3 Theorem 7]{Shafarevich_Basic_algebraic_geometry1}) that $\codim Z_i \geq 2$. 

Another application of (\cite[I 6.3 Theorem 7]{Shafarevich_Basic_algebraic_geometry1}) shows that $\codim \pi_i^{-1} (Z_i) \geq 2$ for every $i$, and therefore $\codim Z \geq 2$. 
\end{proof}

Let $U$ be the subset of $H^0(G, (\Lambda^2 \calS^*)^{2n-2})$ defined $U \colonequals S \setminus p^{-1}(Z)$.

\begin{co}\label{e-well-defined}
A section $\sigma$ of $\calV=\oplus_{i=1}^{2n-2} \wedge^2 \calS^*$ in $U$ has only isolated zeros and $e(\calV, \sigma)$ is independent of the choice of such $\sigma$.
\end{co}

For clarity, we remark that if $\sigma$ is defined over an extension field $E$ of $k$, then $e(\calV, \sigma)$ is an element of $\GW(E)$ and the claimed independence means that $e(\calV, \sigma)$ is the basechange to $E$ of $e(\calV, \sigma')$ for some section $\sigma'$ defined over $k$.

\begin{proof}
A section $\sigma$ of $\calV$ corresponding to a point of $U$ has only isolated zeros because the zeroes of $\sigma = (\sigma_1,\ldots,\sigma_{2n-2})$ are precisely the points in the intersection of $G$ and the hyperplanes corresponding to $\sigma_i$. 

Since $\codim Z \geq 2$, we also have $\codim p^{-1}(Z) \geq 2$, so it follows that $H^0(\calV) - U$ has codimension $\geq 2$. By the proof of  \cite[Lemma 57]{CubicSurface}, the fact that $H^0(\calV) - U$ has codimension $\geq 2$ implies that any two points of $U$ can be connected by affine lines in $U$, after possibly passing to an odd degree field extension. (This latter property, which is described in more detail in \cite[Definition 35 and Corollary 36]{CubicSurface}, is related to $\bbb{A}^1$-chain connectedness as in \cite[Section 2.2]{AsokMorel-smooth_varieties_up_toA1_htpy}.) By \cite[Theorem 3]{CubicSurface},  $e(\calV, \sigma)$  is independent of the choice of $\sigma$ in $U$. \qedhere
\end{proof}

\begin{rmk}
The independence of $e(\calV, \sigma)$ on the choice of $\sigma$ also follows from \cite[Theorem 1.1]{BW-Euler} by a different proof. 
\end{rmk}

\begin{rmk}\label{goodsec}
 Throughout the rest of the paper, we will only work with sections $\sigma$ of $\calV$ in the open set $U$ described above of the form $\sigma = \oplus_{i = 1}^{2n-2} \alpha_i \wedge \beta_i$ for functionals $\alpha_i,\beta_i \in (k^{n+1})^*$ as described in the second paragraph of Section~\ref{vecandsec}. By Corollary~\ref{e-well-defined}, this means we have a well-defined Euler number. 
\end{rmk}

\subsection{A formula for the local index $\ind_L \sigma$}
\begin{pr}\label{indexformula} Let $\sigma$ be the section $\sigma = \oplus_{i = 1}^{2n-2} \alpha_i \wedge \beta_i$ of $\calV$. Suppose that $\sigma$ has an isolated simple zero at the point $L=\Span(e_{n},e_{n+1})$ of $\Gr(2,n+1)$. Let $\alpha_i = \sum_j a_{ij} {\phi}_j$ and $\beta_i = \sum_j b_{ij} {\phi}_j$ be the expansion of the linear forms $\alpha_i$ and $\beta_i$ in terms of the chosen $k$-basis. Then
 \[ \ind_L \sigma = \Bigg\langle \det \begin{bmatrix}
\cdots & (a_{i1}b_{in+1}-a_{in+1}b_{i1}) & \cdots \\ & \vdots & \\ \cdots & (a_{ij}b_{in+1}-a_{in+1}b_{ij}) & \cdots \\ & \vdots & \\ \cdots & (a_{in-1}b_{in+1}-a_{in+1}b_{in-1}) & \cdots \\ \cdots & (a_{in}b_{i1}-a_{i1}b_{in}) & \cdots \\ & \vdots & \\ \cdots & (a_{in}b_{ij}-a_{ij}b_{in}) & \cdots \\ & \vdots & \\ \cdots & (a_{in}b_{in-1}-a_{in-1}b_{in}) & \cdots
\end{bmatrix} \Bigg\rangle  \]
\end{pr}

\begin{proof}

 The line $L$ corresponds to the origin in the affine patch $\Spec k[x_1,\ldots,x_{n-1}, y_1,\ldots, y_{n-1}]$.  Trivialize each of the $2n-2$ summands of $\calV$ in a neighbourhood of $L$ using the section $\tilde{\phi}_{n} \wedge \tilde{\phi}_{n+1}$ of $\calS^* \wedge \calS^*$. Let $(f_1,f_2,\ldots,f_{2n-2})$ be functions on the affine patch $\Spec k[x_1,\ldots,x_{n-1}, y_1,\ldots, y_{n-1}]$  be defined by the relations $\alpha_i \wedge \beta_i = f_i \cdot \tilde{\phi}_n \wedge \tilde{\phi}_{n+1}$. 
 
 By Proposition~\ref{CubicSurfacetrJprop}, it suffices to compute the matrix of partial derivatives of these functions at the origin and take its determinant. First we have the change of basis formulae 
$$  \phi_i = \begin{cases} \tilde{\phi}_i + x_i \tilde{\phi}_n + y_i \tilde{\phi}_{n+1}  &\mbox{for } i = 1,\ldots, n-1 
\\ \tilde{\phi}_n & \mbox{for } i=n
\\ \tilde{\phi}_{n+1} & \mbox{for } i={n+1} \end{cases} .$$
Now 
\begin{align*} 
\alpha_i \wedge \beta_i &= (\sum_j a_{ij} {\phi}_j) \wedge (\sum_j b_{ij} {\phi}_j) \\
&= [ \bigl( \sum_{j=1}^{n-1} a_{ij} (\tilde{\phi}_j + x_j \tilde{\phi}_n + y_j \tilde{\phi}_{n+1}) \big) + a_{in} \tilde{\phi}_n + a_{i n+1} \tilde{\phi}_{n+1} ] \wedge \\
&\qquad [ \bigl( \sum_{j=1}^{n-1} b_{ij} (\tilde{\phi}_j + x_j \tilde{\phi}_n + y_j \tilde{\phi}_{n+1}) \bigr) + b_{in} \tilde{\phi}_n + b_{i n+1} \tilde{\phi}_{n+1} ]
\end{align*}
Since we will be evaluating the matrix of partial derivatives at $e_n \wedge e_{n+1}$, we only need to focus on terms that have $\tilde{\phi}_n \wedge \tilde{\phi}_{n+1}$ or  $\tilde{\phi}_{n+1} \wedge \tilde{\phi}_{n}$. Also, we only need to pick out the linear terms in this expansion, so we may ignore the constant term and higher order terms in the Taylor expansion using the variables $x_i$ and $y_i$. Therefore
\[ \alpha_i \wedge \beta_i = \bigl[\ldots + \sum_j (a_{ij}b_{i n+1}-a_{i n+1}b_{ij}) x_j + \sum_j (a_{in} b_{ij}-a_{ij}b_{in}) y_j + \ldots \bigr] \tilde{\phi}_n \wedge \tilde{\phi}_{n+1}  .\]
Computing partial derivatives and evaluating at $e_n \wedge e_{n+1}$ gives the formula in the statement of the lemma. \qedhere
\end{proof}

\section{A geometric reinterpretation of the local index}\label{LinearAlgebraFormula}
Throughout this section, we will use the local coordinates on $\Gr(2,n+1)$, the orientations of $\Gr(2,n+1)$ and $\calV$, and the compatible local trivialization of $\calV$ introduced in the previous section.

We wish to express the local index at a simple zero $L$ of $\sigma$ in terms of the line $L$ and the configuration of the $\pi_i$. When $k \subseteq k(L)$ is a separable, we do this in terms of the field of definition $k(L)$, the configuration of the intersection points of $\pi_i \cap L$ on $L$, and the configuration of hyperplanes spanned by $\pi_i$ and $L$ in the space of hyperplanes of $\bbb{P}^n$ containing $L$, in the following manner.

Let $W \subset k(L)^{n+1}$ denote the dimension $2$ vector subspace corresponding to the line $L$, so $L$ is canonically isomorphic to $\bbb{P}W$. The space of hyperplanes containing $W$ is canonically identified with the projective space $\bbb{P}(k(L)^{n+1}/W)$. Let $W^*$ denotes the $k(L)$-linear dual of $W$, and similarly for $(k(L)^{n+1}/W)^*$. 

Although the intersection points $L \cap \pi_i$ and the hyperplanes spanned by the $\pi_i$ and $L$ only determine points of $\bbb{P}W$ and $\bbb{P}(k(L)^{n+1}/W)$, respectively, the section $\sigma = \oplus_{i=1}^{2n-2} \alpha_i \wedge \beta_i$ distinguishes points of $W^*$ and $(k(L)^{n+1}/W)^*$. Namely, since $L \cap \pi_i$ is non-empty, we have that the restrictions of $\alpha_i$ and $\beta_i$ to $W$ are linearly dependent, i.e., $\alpha_i \wedge \beta_i$ is in $$\Ker := \Ker ( (k(L)^{n+1})^* \wedge (k(L)^{n+1})^* \to W^* \wedge W^*).$$ There is a natural map \begin{equation}\label{map tokn/W*otimesW*} \Ker \to (k(L)^{n+1}/W)^* \otimes W^*.\end{equation} (To see this, note that the map $(k(L)^{n+1})^* \otimes (k(L)^{n+1}/W)^* \to \Ker/((k(L)^{n+1}/W)^* \wedge (k(L)^{n+1}/W)^*)$ is surjective with kernel $(k(L)^{n+1}/W)^* \otimes (k(L)^{n+1}/W)^* \subset (k(L)^{n+1})^* \otimes (k(L)^{n+1}/W)^*$, giving rise to a natural isomorphism $\Ker/((k(L)^{n+1}/W)^* \wedge (k(L)^{n+1}/W)^*) \cong W^*\otimes (k(L)^{n+1}/W)^*$.)

Choose bases of $W$ and  $k(L)^{n+1}/W$, giving rise to bases of their duals and therefore coordinates of $\bbb{P}W$ and $\bbb{P}(k(L)^{n+1}/W)$. Let $[c_{i0}, c_{i1}]$ be the coordinates of $\pi_i \cap L = [c_{i0}, c_{i1}]$, and let $[d_{i0},d_{i1},\ldots,d_{i,n-2}]$ be the coordinates of the plane spanned by  $\pi_i$ and $L$.

\begin{df}\label{df:normalized_coordinates}
A lift of the homogeneous coordinates $[c_{i0}, c_{i1}]$ and $[d_{i0},d_{i1},\ldots,d_{i,n-2}]$ to coordinates $(c_{i0}, c_{i1})$ and $(d_{i0},d_{i1},\ldots,d_{i,n-2})$ of vectors in $W^*$ and $(k(L)^{n+1}/W)^* $ are {\em normalized} if $(d_{i0},d_{i1},\ldots,d_{i,n-2}) \otimes (c_{i0}, c_{i1})$ is the image of $\alpha_i \wedge \beta_i$ under the map \eqref{map tokn/W*otimesW*}. 
\end{df}

\begin{pr}\label{geometric-interpretation}
Suppose that $L$ is a simple zero of $\sigma$ such that  $k \subseteq k(L)$ is a separable field extension. Pick normalized coordinates for the intersection points $\pi_i \cap L$ and the hyperplanes spanned by $\pi_i$ and $L$ as above, and let
\begin{equation}\label{i(L)def}
i(L) \colonequals \det \begin{bmatrix}
d_{10}c_{10} &\cdots & d_{i0}c_{i0} & \cdots & d_{2n-2,0}c_{2n-2,0}\\ \vdots&& \vdots && \vdots \\ d_{1j}c_{10} &\cdots & a_{ij}d_{i0} & \cdots & d_{2n-2,j}c_{2n-2,0}\\ \vdots & & \vdots & & \vdots \\  d_{1,n-2}c_{10} &\cdots & d_{i,n-2}c_{i0} & \cdots & d_{2n-2,n-2}c_{2n-2,0}\\ d_{10}c_{11}&\cdots & d_{i0}c_{i1} & \cdots & d_{2n-2,0}c_{2n-2,1} \\ \vdots&& \vdots && \vdots \\ d_{1j}c_{11} &\cdots & d_{ij}c_{i1} & \cdots & d_{2n-2,j}c_{2n-2,1}\\ \vdots & & \vdots & & \vdots \\ d_{1,n-2}c_{11} &\cdots & d_{i,n-2}c_{i1} & \cdots & d_{2n-2,n-2}c_{2n-2,1}
\end{bmatrix}
\end{equation} 
Then $\ind_L \sigma = \Tr_{k(L)/k} \langle i(L) \rangle$.
\end{pr}

\begin{proof}
By Proposition \ref{CubicSurfacetrJprop}, it is sufficient to show that the Jacobian determinant $J$ equals $i(L)$, namely $J =i(L)$ in $k(L)^*/(k(L)^*)^2$. In particular, we may assume that $k(L) = k$, and we do this now for notational simplicity. By Lemma \ref{lm:label_oriented_coordinates_Gr}, we may compute $J$ with local coordinates coming from a basis $\{e_1,\ldots,e_{n+1}\}$ such that $L$ corresponds to $W=\Span(e_{n},e_{n+1})$. Since changing the basis of $W$ and $k^n/W$ changes $i(L)$ by a square ($n$ is odd), we may also compute $i(L)$ using the basis $\{e_{n+1}, e_n\}$ of $W$ and the basis of $k^n/W$ determined by the images of $\{e_1,\ldots,e_{n-1}\}.$ This choice determines the normalized coordinates of $\pi_i \cap L$ and the plane spanned by $\pi_i$ and $L$.

Since $\pi_i = \{\alpha_i = 0, \beta_i = 0 \}$ intersects $L$, there is a unique $v_i$ in $W$ such that $\alpha_i(v_i) = \beta_i (v_i) = 0$. It follows that that the vectors $(a_{in}, a_{i,n+1})$ and $(b_{in}, b_{i,n+1})$ are linearly dependent in the linear dual $W^*$ of $W$. In particular, by replacing $(\alpha_i, \beta_i)$ with either $(\alpha_i - c \beta_i, \beta_i)$ for some constant $c$ or $(-\beta_i, \alpha_i)$, we may assume that $(a_{in}, a_{i,n+1}) = (0,0)$.

By Proposition \ref{indexformula}, it follows that \begin{equation}\label{indLsimgafirstdet}\ind_L \sigma = \left\langle \det \begin{bmatrix}
\cdots & a_{i1}b_{in+1} & \cdots \\ & \vdots & \\ \cdots & a_{ij}b_{in+1} & \cdots \\ & \vdots & \\ \cdots & a_{in-1}b_{in+1}& \cdots \\ \cdots & -a_{i1}b_{in}& \cdots \\ & \vdots & \\ \cdots & -a_{ij}b_{in} & \cdots \\ & \vdots & \\ \cdots & -a_{in-1}b_{in} & \cdots
\end{bmatrix} \right\rangle.\end{equation} 

Note that $[b_{i,n+1}, -b_{in}]$ are coordinates in $\bbb{P}W$ for the hyperplane $\{ \alpha_i = \beta_i = 0\} \cap W$ in $W$; $[a_{i1}, a_{i2}, \ldots, a_{i n-1}]$ are coordinates in $\bbb{P}(k^n/W)$ for the plane spanned by $\pi_i$ and $L$; and that $(b_{i,n+1}, -b_{in})$ and $(a_{i1}, a_{i2}, \ldots, a_{i n-1})$ are the normalized coordinates for our chosen bases. \qedhere 
\end{proof}

\section{The Euler number of $\mathcal{V}$}\label{EulercharV}

Let $X$ be a smooth, proper scheme of dimension $n$ over $k$. Let $\calE \to X$ and $\calE' \to X$ be vector bundles of ranks $r$ and $r'$ such that $r+r' =n$ and $\calE \oplus \calE'$ is relatively orientable. Suppose that $\sigma$ and $\sigma'$ are global sections of $\calE$ and $\calE'$ respectively such that the section $\sigma \oplus \sigma'$ has only have isolated zeros admitting Nisnevich local coordinates. The hypothesis that the zeros of $\sigma \oplus \sigma'$ admit Nisnevich coordinates is automatic when $n \geq 1$ by \cite[Proposition 2.23]{BW-Euler}.

\begin{pr}\label{factoringalpha} If $r'$ is odd, then $e(\calE \oplus \calE') := e(\calE \oplus \calE', \sigma \oplus \sigma')$ is an integer multiple of the hyperbolic form $h = \langle -1 \rangle + \langle 1 \rangle$ in $\GW(k)$.
\end{pr}

\begin{rmk}\label{rmk:hypothesis_timing}
When this paper was originally posted to ArXiv, the additional hypothesis that $e(\calE \oplus \calE', \sigma \oplus \alpha \sigma')$ is independent of $\alpha$ for all $\alpha \in k^*$ appeared in Proposition~\ref{factoringalpha}. Note that if $\sigma \oplus \sigma'$ is in the set $U$ in Corollary~\ref{e-well-defined}, so is $\sigma \oplus \alpha \sigma'$ for any $\alpha \in k^*$. By \cite[Theorem 1.1]{BW-Euler} this hypothesis is always satisfied, so it is now omitted, however, the original version was sufficient for the results of this paper. 
\end{rmk}

This proof is extracted from M. Levine's argument that the Euler characteristic of an odd dimensional scheme is a multiple of $h$ \cite[Theorem 7.1]{Levine-EC}. M. Levine also credits J. Fasel. Since the result and the context of the definitions is different, we give the proposition and proof.

\begin{rmk}\label{GWfacts} We will need various facts about the ring $GW(k)$ which we recall first for the convenience of the reader.
\begin{enumerate}
 \item The ring $\GW(k)$ is the zeroth graded summand of the graded ring $K^{\MW}(k)$ introduced by Morel, and then refined in joint work with Hopkins, presented in \cite[Chapter 3]{morel}. 
 \item The graded ring $K^{\MW}(k)$ has generators $[u]$ of degree $1$ for $u$ in $k^*$, and $\eta$ of degree $-1$. The element $\langle u \rangle$ in $\GW(k)$ corresponds to the element $1+ \eta[u]$.
 \item\label{K-1iso} There is a canonical isomorphism $K^{\MW}_{-1}(k) \cong W(k)$ \cite[Lemma 3.10]{morel}. Under this isomorphism $\eta f$ in $K^{\MW}_{-1}(k) $ corresponds to the image of $f$ under the quotient map $\GW(k) \to  \GW(k)/ \mathbb{Z} h \cong \W(k)$ for any $f$ in $\GW(k)$. 
\item\label{partialtmap} There is a map $$\partial_t: K^{\MW}_0(k(t)) \to K^{\MW}_{-1}(k) $$ corresponding to the local ring $k[t]_{(t)}$ with uniformizer $t$ \cite[Theorem 3.15]{morel}. This map has the following two properties:\begin{enumerate}
\item \label{partialtint} For any class $f$ in $\GW(k)$, the class $\partial f_{k(t)} \in \GW(k(t))$ is $0$. 
\item \label{partialt<t>} For any class $f$ in $\GW(k)$, the class $\langle t \rangle f_{k(t)}$ in $\GW(k(t)) \cong K^{\MW}_0(k(t))$ has image $\eta f$ in $K^{\MW}_{-1}(k(t)) $ under $\partial_t$, because $$\partial_t (\langle t \rangle f_{k(t)}) = \partial_t (1+ \eta [t]) f_{k(t)}) = \partial_t (f_{k(t)}) + \partial_t (\eta [t] f_{k(t)}) = \eta \partial_t ([t] f_{k(t)}) = \eta f.$$ Here we are considering $\GW(k)$ as the subring $\GW(\mathbb{A}^1_k) \subset \GW(k(t))$ of unramified elements of $\GW(k(t))$, or equivalently, we are identifying $\GW(k)$ with its image under the pullback map $\GW(k) \to \GW(k(t))$.
\end{enumerate} 
\end{enumerate}
 
\end{rmk}

\begin{proof}[Proof of Proposition~\ref{factoringalpha}]

Note that the set of isolated zeros of $\sigma \oplus \alpha \sigma'$ does not depend on $\alpha$ for any $\alpha$ in $k^*$. Let $x$ be such a zero of $\sigma \oplus \sigma'$. We first show that \begin{equation}\label{ind_scales_with_section}\ind_x (\sigma \oplus \alpha \sigma') = \langle \alpha \rangle^{r'}\ind_x (\sigma \oplus \sigma')\end{equation}  for any $\alpha \in k^*$ and then use this to show that $e$ is an integer multiple of $h$..

When $x$ is a simple zero with separable residue field, the local index $\ind_x (\sigma \oplus \sigma') $ is computed by Proposition~\ref{CubicSurfacetrJprop}, which was originally \cite[Proposition 32]{CubicSurface}. Recall that this computation is accomplished by considering the section $\sigma \oplus \sigma'$ locally to be a function $\bbb{A}^n \to \bbb{A}^n$ and the index is the Jacobian determinant. Replacing $\sigma \oplus \sigma'$ by $\sigma \oplus \alpha \sigma'$ has the effect of multiplying the last $r'$ coordinate projections of the associated function $\bbb{A}^n \to \bbb{A}^n$ by $\alpha$, which in turn scales the Jacobian determinant that computes the local index by $\alpha^{r'}$.

In general, choose an open neighborhood $U$ of $x$ with Nisnevich local coordinates $\phi: U \to \Spec k[x_1, x_2, \ldots, x_n]$ near $x$. Choose local trivializations $\psi: \calE \vert_U \to \calO_U^r$ and $\psi': \calE' \vert_U \to \calO_U^{r'}$ such that $\psi \oplus \psi'$ and $\phi$ are compatible with the relative orientation. Then $\psi \sigma = (f_1, \ldots, f_r)$ and $\psi \sigma' = (f_1', \ldots, f_r')$ where $f_i$ and $f_i'$ are in $\calO(U)$. We choose $g_i$ for $i = 1,\ldots, r$ in $k[x_1, \ldots, x_n]$ and $g_i'$ in $k[x_1, \ldots, x_n]$ for $i = 1,\ldots, r'$  such that $f_i - \phi^* g_i$ and $f_i' - \phi^* g_i'$ are in a sufficiently high power of $m_x$. For notational convenience, define $(h_1, \ldots h_{r+r'})$ by $(h_1, \ldots h_{r+r'})=(g_1, \ldots, g_r, g_1',\ldots,g_r')$. Then $(h_1, \ldots h_{r+r'})$ defines the complete intersection $$\calO_{Z,x} \cong k[x_1,\ldots,x_n]_{m_{\phi(x)}}/\langle h_1, \ldots, h_{r+r'} \rangle$$ and $\ind_x (\sigma \oplus \alpha \sigma')$ is represented by the bilinear pairing on $\calO_{Z,x}$ $$(a,b) \mapsto \zeta(ab),$$ where $\zeta$ is the $k$-linear map $\zeta: \calO_{Z,x} \to k$ constructed in \cite[p. 182]{scheja}. (Note that \cite{scheja} use $\eta$ for their linear map, which we have replaced with $\zeta$ to avoid confusion with the Hopf element $\eta \in K^{\MW}(k)$.) 

The map $\zeta$ is defined as follows. Choose $a_{ij}$ in $k[x_1, \ldots, x_n] \otimes_k k[x_1, \ldots, x_n]$ such that $h_j \otimes 1 - 1 \otimes h_j = \sum_{ij} a_{ij}(x_i \otimes 1 - 1 \otimes x_i)$. Define $\Delta$ in $\calO_{Z,x} \otimes_k \calO_{Z,x}$ to be the image of $\det(a_{ij})$. Let $\chi$ be the canonical map $$\chi: \calO_{Z,x} \otimes_k \calO_{Z,x} \to  \Hom_k(\Hom_k(\calO_{Z,x}, k), \calO_{Z,x})$$ $$ a \otimes a' \mapsto (\gamma \mapsto \gamma(a) a').$$  The map $\zeta$ is the inverse image of $1$ under $\chi(\Delta)$. It is a result of Scheja and Storch that this is well-defined. In particular, if $\Delta$ is multiplied by $\alpha$ in $k^*$, then $\zeta$ is multiplied by $\alpha^{-1}$, and this is the only property of $\zeta$ we will need.

Note that $\Delta$ is multiplied by $\alpha$ when $h_i$ replaced by $\alpha h_{i}$ for some fixed $i$. Changing the section $\sigma \oplus \sigma'$ to $\sigma \oplus \alpha \sigma'$ changes each $f_i'$ to $\alpha f_i'$ and leaves the $f_i$ fixed. We may therefore choose new $g_i$ and $g_i'$ by leaving the $g_i$ fixed and changing $g_i'$ to $\alpha g_i'$. Therefore new $a_{ij}$ can be defined by keeping $a_{ij}$ the same for $j \leq r$ and changing $a_{ij}$ to $\alpha a_{ij}$ for $j = r+1, \ldots, r+r'$. Therefore the element $\Delta_{\alpha}$ for $\sigma \oplus \alpha \sigma'$ is $\Delta_{\alpha} = \alpha^{r'}\Delta$. Thus the map $\zeta_{\alpha}$ for the section $\sigma \oplus \alpha \sigma'$ is $\zeta_{\alpha} = \alpha^{-r'} \zeta$, giving \eqref{ind_scales_with_section}.

For a field extension $k \subseteq L$, the Euler number $e(\calE_L \oplus \calE'_L)$ in $\GW(L)$ is the pullback of $e(\calE \oplus \calE')$ in $\GW(k)$ by functoriality, i.e. $e(\calE_L \oplus \calE'_L) = e(\calE \oplus \calE') \otimes_k L$.

Furthermore, for any $\alpha$ in $L^*$, we have $e(\calE_L \oplus \calE'_L) = e(\calE_L \oplus \calE'_L, \sigma \oplus  \alpha \sigma')$ in $\GW(L)$ by \cite[Theorem 1.1]{BW-Euler}. See also Remark \ref{rmk:hypothesis_timing}. Thus $$e(\calE_L \oplus \calE'_L)=e(\calE_L \oplus \calE'_L, \sigma \oplus  \alpha \sigma') = \sum_{x: \sigma(x)=0, \sigma'(x)=0} \ind_x (\sigma \oplus \alpha \sigma').$$ Since $r'$ is odd, we have $$\ind_x (\sigma \oplus \alpha \sigma') = \langle \alpha \rangle^{r'}\ind_x (\sigma \oplus \sigma') = \langle \alpha \rangle \ind_x  (\sigma \oplus \sigma'),$$ by \eqref{ind_scales_with_section}. It follows that \begin{equation}\label{<a>e=e} e(\calE_L \oplus \calE'_L)=  \langle \alpha \rangle e(\calE_L \oplus \calE'_L).\end{equation} for all $\alpha$ in $L^*$.

To simplify notation, let $e =  e(\calE \oplus \calE')$ in $\GW(k)$ and let $e_L = e \otimes_k L$ be the pullback to $\GW(L)$. So \eqref{<a>e=e} says that $\langle \alpha \rangle e_L = e_L$ for all $\alpha$ in $L^*$.

We now show that $e$ is an integer multiple of $h$. By \cite[IX Milnor's Theorem 3.1 p. 306]{lam05}, it suffices to show that $e_{k(t)}$ is an integer multiple of $h$. By Remark~(\ref{GWfacts})(\ref{partialtmap}), it follows that 
$$0 = \partial_t (e_{k(t)} ) = \partial_t (\langle t \rangle e_{k(t)}) = \eta (e),$$ where the first equality is \eqref{partialtint}, the second is  \eqref{<a>e=e}, and the third is \eqref{partialt<t>}. By Remark~(\ref{GWfacts})(\ref{K-1iso}) $e_{k(t)}$ is $0$ in $\W(k)$, whence a multiple of $h$ and the claim follows.  

\end{proof}

\begin{co}\label{enrichedCatalan}
 \[  e(\calV) = \frac{1}{2} \frac{(2n-2)!}{n! (n-1)!} h .\]
\end{co}
\begin{proof}
By Proposition \ref{factoringalpha}, $e(\calV)$ is a multiple of $h$. By Lemma \ref{lm:rk_sgn_e}, the rank of $e(\calV)$ is the topological Euler number $e_{\C}$ of $\calV(\C)$. This topological Euler number is $\tfrac{(2n-2)!}{n! (n-1)!}$. See for example \cite[Proposition~4.12]{EisenbudHarris}.
\end{proof}

We now give an alternate, more explicit proof of this calculation when $n=4$ using the local index calculations from Section~\ref{LinearAlgebraFormula}. The explicit form of the calculation will be used in Theorem~\ref{thmcrossratio}.

\begin{pr}\label{explicitEulernumber}
 Let $\calS$ be the tautological bundle of $\Gr(2,4)$. Let $\calV = \oplus_{i=1}^{4} \Lambda^2(\calS^*)$. Then $e(\calV) = h$.
\end{pr}
\begin{proof}
 As in Section~\ref{LinearAlgebraFormula}, we will compute this Euler characteristic by adding up local contributions from the explicit section $\sigma$ of $\calV$ coming from four lines in $\bbP^3_k$. By Corollary \ref{e-well-defined}, we may choose a $k$-rational section $\sigma$ with only isolated zeros for this computation. We may moreover assume that the two lines corresponding to the two zeroes of $\sigma$ are $k$-rational, by explicit example (see for instance Section~\ref{example}). Classical arguments (see for instance \cite[Section~3.4.1]{EisenbudHarris}) show that these two lines can be taken to be two skew lines in the same ruling of a quadric surface in $\bbP^3_k$, so after a change of coordinates, we 
may assume that the two zeroes of $\sigma$ are the lines corresponding to the subspaces $L' \colonequals \Span(e_1,e_2)$ and $L \colonequals \Span(e_3,e_4)$ of $k^4$.
 
 We will now show that a careful choice of $\alpha_i,\beta_i$ (which in turn determine the section $\sigma$), and local coordinates around $L'$ and $L$ lets us show that the matrices computing $\ind_{L'} \sigma$ and $\ind_{L} \sigma$ are related by row operations and sign swaps.
 
 For any index $i$, since $L = \Span(e_3,e_4)$ intersects $L_i$, it follows that $\{ \phi_1,\phi_2,\alpha_i,\beta_i \}$ are linearly dependent. In terms of the chosen basis expansions for $\alpha_i$ and $\beta_i$, this translates to
 \[ \det \begin{bmatrix} a_{i3} & a_{i4} \\ b_{i3} & b_{i4}  \end{bmatrix} = 0 .\]
 Since we may replace the pair $(\alpha_,\beta_i)$ by any pair of linearly independent vectors in $\Span(\alpha_i,\beta_i)$, by either adding a multiple of $\beta_i$ to $\alpha_i$ or by swapping $\alpha_i$ and $\beta_i$, we may assume that $a_{i3} = a_{i4} = 0$ without any loss of generality.
 Similarly from the condition that $L' = \Span(e_1,e_2)$ intersects $L_i$, it follows that \[ \det \begin{bmatrix} a_{i1} & a_{i2} \\ b_{i1} & b_{i2}  \end{bmatrix} = 0 ,\]
 and as before, we can use this relation and further change $\alpha_i,\beta_i$ to assume that $b_{i1} = b_{i2} = 0$.
 
 With these choices, we get 
 \[ \ind_{L} \sigma = \left\langle \det \begin{bmatrix}
 \cdots & a_{i1}b_{i4} & \cdots & \\ \cdots & a_{i2}b_{i4} & \cdots \\  \cdots & -a_{i1}b_{i3} & \cdots \\ \cdots & -a_{i2}b_{i3} & \cdots
\end{bmatrix} \right\rangle .\]

To compute $\ind_{L'} \sigma$, we need to choose local coordinates in $\Gr(2,4)$ around this point, and a local trivialization of $\Lambda^2(\calS^*)$ compatible with the chosen orientation. By Lemma \ref{lm:label_oriented_coordinates_Gr}, if we set \[ f_1=e_3,f_2=e_4, f_3=e_1,f_4=e_2, \]
and define $\{\tilde{f}_{1}, \tilde{f}_{2},\tilde{f}_{3},\tilde{f}_{4} \}$ by the formula $$\tilde{f}_i = \begin{cases} f_i  &\mbox{for } i = 1,2 \\ 
 x_1 f_1 + x_2 f_2 + f_{3} & \mbox{for } i=3
\\ 
y_1 f_1 + y_2 f_2 + f_{4} & \mbox{for } i=4, \end{cases} $$
then $x_1,x_2,y_1,y_2$ gives us local coordinates around $L' = \Span(f_3,f_4) = \Span(e_1,e_2)$.
Let $\{\tilde{\psi}_{1},\tilde{\psi}_{2},\tilde{\psi}_{3}, \tilde{\psi}_{4} \}$ be the dual basis for the basis $\{\tilde{f}_{1}, \tilde{f}_{2}, \tilde{f}_{3},\tilde{f}_{4} \}$ of $k^{n+1}$. Then as before it follows that $\tilde{\psi}_3 \wedge \tilde{\psi}_4$ gives a trivialization of $\Lambda^2(\calS^*)$ around $\Span(e_1,e_2)$ compatible with the chosen orientation and hence also a trivialization of $\calV$. With these choices, if we now redo the index calculation in Proposition~\ref{indexformula}, the formula~\ref{indLsimgafirstdet} simplifies to

\[ \ind_{L'} \sigma = \left\langle \det \begin{bmatrix} \cdots & -a_{i2}b_{i3} & \cdots \\ \cdots & -a_{i2}b_{i4} & \cdots \\   \cdots & a_{i1}b_{i3} & \cdots \\
 \cdots & a_{i1}b_{i4} & \cdots & 
\end{bmatrix} \right\rangle.\]
From these explicit formulae, we see that the matrices computing $\ind_{L'} \sigma$ and $\ind_{L} \sigma$ are related by swapping the first and fourth rows, and by multiplying each of the second and third rows by $-1$. This in turn implies that if $\ind_{L'} \sigma = \langle c \rangle$, then $\ind_{L} \sigma = \langle -c \rangle$, and therefore $e(\calV) = \langle c \rangle + \langle -c \rangle = \langle 1 \rangle + \langle -1 \rangle$.
\end{proof}

\section{The local index in terms of cross ratios}\label{crossratio}

\begin{df}\label{defcrossratio}
 Let $L_1,L_2,L_3,L_4$ be four lines in $\bbP^3_k$ and let $L$ be another line in $\bbP^3$ that meets all four lines. Assume that the four intersection points of $L$ with the $L_i$ are distinct. Also assume that the four planes spanned by $L$ and the $L_i$ are distinct. Let $\lambda_{L}$ be the cross ratio of the four points $L \cap L_1,L \cap L_2, L \cap L_3, L \cap L_4$ on the line $L$. Let $\mu_L$ be the cross ratio of the four planes containing $L$ and each of the four lines $L_1,L_2,L_3,L_4$ in the $\bbP^1_k$ of planes containing $L$ in $\bbP^3_k$. 
\end{df}

The goal of this section is to prove the following theorem.

\begin{tm}\label{thmcrossratio}
Let $\calV = \oplus_{i=1}^{4} \Lambda^2 \calS^*$. Let $L_1,L_2,L_3,L_4$ be four lines in $\bbP^3_k$ such that
\begin{itemize}
 \item the corresponding section $\oplus_{i=1}^4 \alpha_i \wedge \beta_i$ has only simple zeroes,
 \item for any line $L$ meeting all four lines, the four intersection points $L \cap L_i$ are pairwise distinct, and,
 \item for any line $L$ meeting all four lines, the four planes spanned by $L$ and each of the $L_i$ are pairwise distinct.
\end{itemize}
The locus of such lines is an open dense subset of $\Gr(2,4)^4$. Let $\lambda_L,\mu_L$ be the associated cross ratios as in Definition~\ref{defcrossratio}. There exists a section $\sigma$ of $\calV$ such that for all $L$ meeting $L_1,L_2,L_3,L_4$, we have
 \[ \ind_{L} \sigma = \Tr_{k(L)/k} \langle \lambda_L-\mu_L \rangle \]
\end{tm}

\begin{proof} We first justify that the locus of lines satisfying the conditions above is an open dense subset of $\Gr(2,4)^4$. The locus of sections with simple zeroes is the locus where the corresponding Jacobian determinant appearing in Proposition~\ref{CubicSurfacetrJprop} is nonzero in the residue field, which is an open condition. The condition that the four intersection points $L \cap L_i$ are pairwise distinct is implied by the condition that the four lines are pairwise non-intersecting, which is manifestly an open condition. Finally, note that if the four lines are pairwise non-intersecting, the planes spanned by $L$ and each of the $L_i$ are also distinct, so the third condition follows as well. This means the cross ratios in the statement of the theorem are well-defined and we may use Proposition~\ref{CubicSurfacetrJprop}(originally from \cite{CubicSurface}) to compute the local index. More precisely, it is enough to show that $\ind_{L} \sigma$ computed over the field of definition of the line $k(L)$ equals $\langle \lambda_L-\mu_L\rangle$.
 
 Recall that the line $L_i$ is cut out by the two hyperplanes $\alpha_i = \sum_j a_{ij} {\phi}_j$ and $\beta_i = \sum_j b_{ij} {\phi}_j$. As in the proof of Proposition~\ref{explicitEulernumber}, we may assume that the two lines meeting these four lines are $L = \Span(e_3,e_4)$ and $L' = \Span(e_1,e_2)$ and that $a_{i3} = a_{i4} = b_{i1} = b_{i2} = 0$ without any loss of generality. We have explicit projective coordinates $[0:0:z_3:z_4]$ on $L = \Span(e_3,e_4)$ induced from the projective coordinates $[z_1:z_2:z_3:z_4]$ on $\bbP^3_k$. In these coordinates, the point $L \cap L_i$ is $[0:0:-b_{i4}:b_{i3}]$. Since cross ratios are invariant under automorphisms of $L$ and since $[0:0:z_3:z_4] \mapsto [0:0:-z_4:z_3]$ is an automorphism of $L$, it follows that $\lambda_L$ is the cross ratio of the four points $[b_{i3}:b_{i4}]$ for $i=1,2,3,4$. Similarly, with $L' = \Span(e_1,e_2)$, we have that $\lambda_{L'}$ is the cross ratio of the four points $[a_{i1}:a_{i2}]$.
 
 We will now pick explicit coordinates $[w_1:w_2]$ on the $\bbP^1_k$ of 2-planes containing $L = \Span(e_3,e_4)$ in $\bbP^3_k$ to compute $\mu_L$. The isomorphism is given by mapping $[w_1:w_2]$ to the 2-plane $\Span(e_3,e_4,-w_2e_1+w_1e_2)$ in $\bbP^3$. In these coordinates, the plane containing $L$ and $L_i$, namely $\Span(e_3,e_4,-a_{i2}e_1+a_{i1}e_2)$, is $[a_{i1}:a_{i2}]$. Therefore $\mu_L$ is the cross ratio of the four points $[a_{i1}:a_{i2}]$ for $i=1,2,3,4$. Similarly $\mu_{L'}$ is the cross ratio of the four points $[b_{i3}:b_{i4}]$ for $i=1,2,3,4$.
 
 Pick $A,B \in \GL_2(k)$ such that their respective images $\overline{A},\overline{B} \in \PGL(k)$ satisfy
 \begin{align*} 
 \overline{B}[b_{13}:b_{14}] &= [1:0] \quad \quad \quad &\overline{A}[a_{11}:a_{12}] &= [1:0] \\ 
 \overline{B}[b_{23}:b_{24}] &= [1:1] \quad \quad \quad &\overline{A}[a_{21}:a_{22}] &= [1:1] \\ 
 \overline{B}[b_{33}:b_{34}] &= [0:1] \quad \quad \quad &\overline{A}[a_{31}:a_{32}] &= [0:1] \\ 
 \overline{B}[b_{43}:b_{44}] &= [1:\lambda_L] \quad \quad \quad &\overline{A}[a_{41}:a_{42}] &= [1:\mu_L]. 
 \end{align*}
 Now change coordinates on the underlying $\bbP^3$ using the block matrix $\begin{bmatrix} A & 0 \\0 & B \end{bmatrix}$ in $\GL_4(k)$. In these new coordinates, applying Proposition~\ref{indexformula} we get
 \[ \ind_{L} \sigma = \left\langle c \det \begin{bmatrix}
0 & 1 & 0 & \lambda_L \\ 0 & 1 & 1 & \lambda_L \mu_L \\ -1 & -1 & 0 & -1 \\ 0 & -1 & 0 & -\mu_L
\end{bmatrix} \right\rangle  = \langle c(\lambda_L-\mu_L) \rangle,\]
where $c \in k$ is an overall nonzero constant coming from the fact that we have picked lifts $A$ and $B$ of the elements $\overline{A},\overline{B}$ in $\PGL_2(k)$. We can now eliminate $c$ by scaling the section $\sigma$ by $1/c$, since this scales the local index by the same factor.
 
Fix a section $\sigma$ such that $\ind_L \sigma = \langle \lambda_L-\mu_L \rangle$. We will show that we also have $\ind_{L'} \sigma = \langle \lambda_{L'}-\mu_{L'} \rangle$. This follows from the following two facts.
\begin{itemize}
 \item As justified in the previous paragraphs, replacing $L$ by $L'$ switches the two cross ratios, so we have $(\lambda_{L'} - \mu_{L'}) = (\mu_L-\lambda_L)$.
 \item The proof of  Proposition~\ref{explicitEulernumber} shows that $\ind_{L'} \sigma = \langle -1 \rangle \ind_{L} \sigma$. \qedhere
\end{itemize}
\end{proof}

\section{Proofs of Main Theorems and an example}\label{mainthm_Section}
\begin{proof} [Proof of Theorem~\ref{intro:4linesthm}]
 Corollary~\ref{e-well-defined} tells us that we may compute $e(\calV)$ by making any auxiliary choice of section $\sigma \in H^0(\Gr(2,n+1),\calV)(k) \setminus Z(k)$. We would like to use the explicit section $\sigma = (\alpha_1 \wedge \beta_1, \alpha_2 \wedge \beta_2), \alpha_3 \wedge \beta_3, \alpha_4 \wedge \beta_4)$ arising from four general lines $L_1,L_2,L_3,L_4$ in $\bbP^3_k$ for our computations. To be able to apply Proposition~\ref{explicitEulernumber}, which in turn relies on Proposition~\ref{CubicSurfacetrJprop}, we need the section $\sigma$ to have only simple zeroes. Further, to be able to interpret the local index in terms of cross-ratios as in Theorem~\ref{thmcrossratio}, we need the cross-ratios to be well-defined, which in turn needs the points of intersection of the simple zero $L$ with the four lines $L_i$ to be pairwise distinct, and the planes spanned by the simple zero $L$ and each of the four lines $L_i$ to also be pairwise distinct.  All of these conditions are satisfied by an open 
subset of lines $(L_1,L_2,L_3,L_4) \in \Gr(2,4)^4$. This open locus for instance contains the locus of four lines that are pairwise non intersecting, and such that the fourth line is not tangent to the unique quadric containing the first $3$. The proof of the theorem is now a direct application of Proposition~\ref{explicitEulernumber} and Theorem~\ref{thmcrossratio}. \qedhere
\end{proof}

 The proof of Theorem~\ref{intro:2n-2codim2hyperplanesthm} is similar to the proof of Theorem~\ref{intro:4linesthm}.

\begin{proof}[Proof of Theorem~\ref{intro:2n-2codim2hyperplanesthm}]
For an open subset of the product of Grassmannians parametrizing $2n-2$-tuples of codimension $2$ planes, the corresponding sections $\sigma = \oplus_{i=1}^{2n-2} \alpha_i \wedge \beta_i$ satisfies the condition that its zeroes are all isolated and simple (see Corollary~\ref{e-well-defined} and the definition just before Proposition~\ref{CubicSurfacetrJprop}). If we assume that either that $k$ is perfect or that $k(L)/k$ is separable for all zeroes $L$ of $\sigma$ so that Proposition~\ref{geometric-interpretation}  applies, then we may compute the local index over the field of definition $k(L)$ of the line and then apply $\Tr_{k(L)/k}$ to obtain the local index over $k$. The theorem now follows from Corollary~\ref{e-well-defined} and Corollary~\ref{enrichedCatalan}.
\end{proof}

\begin{proof}[Proof of Corollary~\ref{Fqcorthm4lines}]
By Theorem \ref{intro:4linesthm}, we have that $\Tr_{\bbb{F}_{q^2}/\bbb{F}_{q}} \langle \lambda_L - \mu_L \rangle$ is a square for $q \equiv 1 \mod 4$ and is a non-square for $ q \equiv 3 \mod 4$. Moreover, $\Disc (\Tr_{\bbb{F}_{q^2}/\bbb{F}_{q}} \langle a \rangle)$ is a non-square if $a$ is a square and a square if $a$ is a non-square  by, for example, \cite[II.2.3]{Conner_Perlis}.
\end{proof}

\subsection{Example}\label{example}                                                                                                                  
We conclude the paper with an explicit example of Theorem~\ref{intro:4linesthm}. Let
\begin{align*} L_1 &\colon [t:s] \mapsto [t:0:0:s] \\ L_2 &\colon [t:s] \mapsto [t:s:t:s] \\ L_3 &\colon [t:s] \mapsto [s+t:2s:2s+2t:s] \\ L_4 &\colon [t:s] \mapsto [3t:t:s:3s] \end{align*}
be the parametric equations of 4 lines in $\bbb{P}^3 = \Proj k[x,y,z,w]$. When the characteristic is not $2,3$ or $5$, these lines pairwise do not intersect and the first three lines lie on the quadric $xy-zw$. (In characteristic $3$ the lines $L_2$ and $L_3$ intersect, and we can instead work with the example where we replace $L_3$ alone by $[-(s+t):-s:s+t:s]$. Similarly, in characteristic $5$ the lines $L_3$ and $L_4$ intersect, and we can replace $L_4$ alone by $[-2t:t:2s:-s]$.) The parametric equations of the lines that meet all four lines are $L \colon [t:s] \mapsto [s:t:t:s]$ and $L' \colon [t:s] \mapsto [-s:t:-t:s]$, with
\begin{align*} 
 L_1 \cap L &= [1:0:0:1] \leftrightarrow [0:1] \quad \quad \quad &L_1 \cap L' &= [-1:0:0:1] \leftrightarrow [0:1] \\ 
 L_2 \cap L &= [1:1:1:1] \leftrightarrow [1:1] \quad \quad \quad &L_2 \cap L' &= [-1:1:-1:1] \leftrightarrow [1:1] \\ 
 L_3 \cap L &= [1:2:2:1] \leftrightarrow [2:1] \quad \quad \quad &L_3 \cap L' &= [-1:2:-2:1] \leftrightarrow [2:1] \\ 
 L_4 \cap L &= [3:1:1:3] \leftrightarrow [1:3] \quad \quad \quad &L_4 \cap L' &= [-3:-1:1:3] \leftrightarrow [-1:3]. 
 \end{align*}
The formula for the cross ratio of four points $z_1,z_2,z_3,z_4$ in $k$ is $\frac{(z_4-z_1)(z_2-z_3)}{(z_2-z_1)(z_4-z_3)}$. This leads to $\lambda_L = 1/3$ and $\lambda_{L'} = -1/5$. 

Now we choose explicit projective coordinates on the $\bbP^1_k$ of planes containing $L$ and the $\bbP^1_k$ of planes containing $L'$ in order to compute the cross ratios $\mu_{L}$ and $\mu_{L'}$. For the former, let $[t:s]$ be the coordinates of the $2$-plane $\Span ((1,0,0,1),(0,1,1,0),(t,s,0,0))$ in $\bbP^3$, and for the latter let it be the coordinates of the $2$-plane $\Span ((-1,0,0,1),(0,-1,1,0),(t,s,0,0))$. In these coordinates, we have
\begin{align*} 
 \Span(L,L_1) &= [1:0] \quad \quad \quad &\Span(L',L_1) &= [1:0] \\ 
 \Span(L,L_2) &= [1:-1]  \quad \quad \quad &\Span(L',L_2) &= [1:1] \\ 
 \Span(L,L_3) &= [1:-2] \quad \quad \quad &\Span(L',L_3) &= [1:2] \\ 
 \Span(L,L_4) &= [3:1] \quad \quad \quad &\Span(L',L_4) &= [3:1],. 
 \end{align*}
which leads to $\mu_L = -1/5 = \lambda_{L'}$ and $\mu_{L'} = 1/3 = \lambda_L$, and $\lambda_L-\mu_L = 8/15 = -( \lambda_{L'}-\mu_{L'})$. Finally we have $\langle 8/15 \rangle + \langle -8/15 \rangle = \langle 1 \rangle + \langle -1 \rangle$ in $\GW(k)$.

\appendix

%
%
%
%

\section{Enriched Local Indices for Non-Rational Sections \\\small by Borys Kadets, Padmavathi Srinivasan, Ashvin A.~Swaminathan, \\ Libby Taylor, and Dennis Tseng}
\subsection{Motivation}
To illustrate the goal of this section, consider the following example. Suppose we have four distinct lines $\pi_1,\ldots,\pi_4\in\bbb{P}^3$ such that the union $\pi_1\cup\cdots\cup \pi_4$ is defined over our base field $k$, but each individual line $\pi_i$ may not be. Equivalently, we have an \'{e}tale $k$-algebra $E$ of degree 4 over $k$, and a map $\Spec E\to \Gr(2,4)$ over $k$. We could extend $k$ and apply Theorem 1, but we would like to leverage the fact that the map $\Spec E\to \Gr(2,4)$ is defined over $k$ and obtain an invariant in $\GW(k)$.

As before, we let $\mathcal{S}$ be the tautological subbundle on $\Gr(2,4)$ and let $\mathcal{L}=\mathcal{S}^{*}\wedge \mathcal{S}^{*}$. Unlike in Section 2, we no longer have a canonical section of $\mathcal{L}^{\oplus 4}$ associated to our lines and our choice of their defining equations. However, the key idea is to notice that we do have a canonical section of the bundle  $\Res_{E/k} \mathcal{L}$, which is \emph{a priori} a twist of $\mathcal{L}^{\oplus 4}$ but in fact turns out to be isomorphic to $\mathcal{L}^{\oplus 4}$ over $k$, and we can compute the enriched Euler number of this canonical section by expressing it as a sum of local indices, just as was done in the proofs of Theorems~1 and~2. In what follows, we compute these local indices in a very general setting.

\subsection{Setup}

Let $k$ be a field, let $\ksep$ be a fixed separable closure of $k$, and let $E$ be an \'{e}tale $k$-algebra of degree $m$. Let $\calV$ be a vector bundle of rank $r$ on a smooth $k$-scheme $X$ of dimension $mr$ equipped with a relative orientation, and let $\sigma \in \calV_E(X_E)$ be a section.

\begin{df}
 Let $\Res_{E/k} \calV$ be the vector bundle (defined over $k$) whose sections on an open set $U$ are given by $\calV_E(U_E)$.
\end{df}
The section $\sigma$ induces a global section $\sigma_{\Res}$ of $\Res_{E/k} \calV$. There is a natural homomorphism $\varphi \colon E \rightarrow \End_k(\Res_{E/k} \calV)$ sending $e \in E$ to the map of multiplication by $e$, and an embedding $\tau\colon \calV \to \Res_{E/k} \calV$. We fix a choice of $k$-basis $\alpha_1 = 1, \alpha_2, \ldots,\alpha_{m}$ of $E$, which determines an isomorphism $\calV^m \rightarrow \Res_{E/k} \calV$ given by Lemma \ref{Isomor}.
\begin{lm}\label{Isomor}
	The map $\calV^m \rightarrow \Res_{E/k} \calV$ defined on sections by $(s_1,\ldots,s_m) \mapsto \sum_q \varphi(\alpha_{q}) \tau(s_q)$  is an isomorphism of vector bundles over $k$.
\end{lm}
\begin{proof}
	This map is well-defined and is readily checked to be an isomorphism on fibers.
\end{proof}
We assume that the section $(\sigma_1, \dots, \sigma_m) \in \calV(X)^m$ corresponding to $\sigma$ via the isomorphism in Lemma~\ref{Isomor} has a 0-dimensional vanishing scheme that is \'etale over $k$. Let $P \in X$ be one such zero having residue field $k(P)$ over $k$. Let $j_1, \dots, j_m\colon E \to \ksep$ be the geometric points of $E$ over $k$. Each $j_i$ induces a map $j_i\colon  (\Res_{E/k}\calV)(X) \to \calV_\ksep(X_\ksep)$, and \mbox{the map $j_i \circ \tau \colon \calV(X) \to \calV_\ksep(X_\ksep)$} does not depend on $i$. From the proof of Proposition \ref{disclocind}, the condition on $(\sigma_1, \dots, \sigma_m)$ having isolated simple zeros is equivalent to $(j_1\circ \sigma,\ldots,j_m\circ \sigma)\in \calV_{\ksep}(X)^m$ having isolated simple zeroes.

\subsection{Local indices of the section of $\Res_{E/k} \calV$}
 The following result gives an explicit formula for the local index of $\sigma_{\Res} \in (\Res_{E/k} \calV)(X) $; the answer is related to the Jacobian determinant at $P$ of the section $(j_1(\sigma), \dots, j_m(\sigma))$ of $\calV_\ksep^m$ by an explicit factor.

\begin{pr}\label{disclocind}
	Let $J(\sigma) \in \GL_{mr}(\ksep)$ be the Jacobian matrix for the map $\bbb{A}^{mr} \to \bbb{A}^{mr}$ induced by $(j_1(\sigma), \dots, j_m(\sigma)) \in \calV_\ksep^m(X_\ksep)$ with respect to a trivialization of $\calV^m$ and local Nisnevich coordinates in an open neighborhood of $P$ that are compatible with the relative orientation over $k$. Then
	\[ \ind_P (\sigma_{\Res}) = \Tr_{k(P)/k}\big(\langle (\det A)^{-r} \cdot \det J(\sigma) \rangle \big)  \in \GW(k),\]
where $A$ is the $n \times n$ matrix whose row-$p$, column-$q$ entry is $j_p(\alpha_q)$.
\end{pr}

\begin{rmk}
The determinant $\det A$ may not be defined over $k$, but the product $(\det A)^{-r} \cdot \det J(\sigma)$ is. Also, $(\det A)^2 \in k^\times/k^{\times 2}$ is the discriminant of the minimal polynomial of a generator of $E/k$; moreover, $(\det A)^2$ is used to define the relative discriminant in the case $k$ is a number field and $E$ is a field extension.
\end{rmk}

\begin{proof}[Proof of Proposition~\ref{disclocind}]
Let $(\sigma_1,\ldots,\sigma_m) \in \calV(X)^m$ be the section corresponding to $\sigma_{\Res} \in (\Res_{E/k}\calV)(X)$ under the isomorphism $\calV^m \to \Res_{E/k}\calV$ of Lemma~\ref{Isomor}. So by definition $\sigma_{\Res} = \sum_q \varphi(\alpha_{q})\tau(\sigma_q)$ and $\sigma = \sum_q \alpha_{q}\sigma_q$, where we regard the $\sigma_i$ as sections of $\calV_E$ under the inclusion $k\hookrightarrow E$. Therefore,
\begin{align*}
    \begin{pmatrix}
    j_1(\sigma) \\ \vdots \\ j_m(\sigma)
    \end{pmatrix}
    &=
    \begin{pmatrix}
    \sum_q j_1(\alpha_{q})\sigma_q \\ \vdots \\ \sum_q j_m(\alpha_{q})\sigma_q)
    \end{pmatrix}
    =
    \begin{pmatrix}
    j_1(\alpha_1) & \cdots & j_1(\alpha_m)\\
    \vdots & \ddots & \vdots \\
    j_m(\alpha_1) & \cdots & j_m(\alpha_m)
    \end{pmatrix}
    \begin{pmatrix}
    \sigma_1 \\ \vdots \\ \sigma_m
    \end{pmatrix}
    \end{align*}
    showing
    \begin{align}
    \label{relatetwo}
            \begin{pmatrix}
    \sigma_1 \\ \vdots \\ \sigma_m
    \end{pmatrix}&=
         \begin{pmatrix}
    j_1(\alpha_1) & \cdots & j_1(\alpha_m)\\
    \vdots & \ddots & \vdots \\
    j_m(\alpha_1) & \cdots & j_m(\alpha_m)
    \end{pmatrix}^{-1}
     \begin{pmatrix}
    j_1(\sigma) \\ \vdots \\ j_m(\sigma)
    \end{pmatrix}
\end{align}
The matrix in \eqref{relatetwo} is actually an $mr\times mr$ matrix with blocks of size $r\times r$. The $ij$ block is the $r\times r$ identity matrix times $j_{i}(\alpha_j)$.

By the computation of the local index in \cite[Proposition 32]{CubicSurface} using the Jacobian, we have that $\ind_P(\sigma_{\Res})$ is the $\Tr_{k(P)/k}\big(\langle \det J(\sigma_1,\ldots,\sigma_m)\rangle \big)$, where $J(\sigma_1,\ldots,\sigma_m)$ is the Jacobian determinant of $(\sigma_1,\ldots,\sigma_m)\colon \bbb{A}^{mr}\to \bbb{A}^{mr}$. By the chain rule and \eqref{relatetwo}, we find that
\begin{align*}
\ind_P(\sigma_{\Res}) & = \Tr_{k(P)/k}\big(\langle \det J(\sigma_1,\ldots,\sigma_m) \rangle \big) = \Tr_{k(P)/k}\big(\langle \det A^{-r} \cdot \det J(\sigma) \rangle)  \\
& = \Tr_{k(P)/k}\big(\langle (\det A)^{-r} \cdot \det J(\sigma) \rangle) \in \GW(k). \qedhere
\end{align*}	
\end{proof}

\subsection{Acknowledgements}
We warmly thank Leonardo Constantin Mihalcea for suggesting an arithmetic count of the lines meeting four lines in space after a talk by the second named author on \cite{CubicSurface}.

Kirsten Wickelgren was partially supported by National Science Foundation Awards DMS-1552730 and DMS-2001890. She also wishes to thank Universit\"at Regensburg and the Newton Institute for hospitality while writing this paper. 

The authors of the appendix would like to thank the organizers of the Arizona Winter School for the opportunity to work on this project, and Kirsten Wickelgren for her mentorship.
}
\bibliographystyle{amsalpha}

%
\bibliography{FourLines}

\begin{bibdiv}
\begin{biblist}

\bib{AsokFasel_comp_euler_classes}{article}{
   author={Asok, A.},
   author={Fasel, J.},
   title={Comparing Euler classes},
   journal={Q. J. Math.},
   volume={67},
   date={2016},
   number={4},
   pages={603--635},
   issn={0033-5606},
   review={\MR{3609848}},
   doi={10.1093/qmath/haw033},}


\bib{AsokMorel-smooth_varieties_up_toA1_htpy}{article}{
      author={Asok, A.},
      author={Morel, F.},
       title={Smooth varieties up to {$\Bbb A^1$}-homotopy and algebraic
  {$h$}-cobordisms},
        date={2011},
        ISSN={0001-8708},
     journal={Adv. Math.},
      volume={227},
      number={5},
       pages={1990\ndash 2058},
         url={https://doi.org/10.1016/j.aim.2011.04.009},
      review={\MR{2803793}},
}

\bib{BW-Euler}{unpublished}{
      author={Bachmann, T.},
      author={Wickelgren, K.},
       title={{$\Bbb A^1$}-Euler classes: six functors formalisms, dualities, integrality and linear subspaces of complete intersections},
        date={2020},
        note={{\em Preprint}, available at
  \url{https://arxiv.org/abs/2002.01848}},
}

\bib{bhatwadekar06}{article}{
      author={Bhatwadekar, S.~M.},
      author={Das, M.~K.},
      author={Mandal, S.},
       title={Projective modules over smooth real affine varieties},
        date={2006},
        ISSN={0020-9910},
     journal={Invent. Math.},
      volume={166},
      number={1},
       pages={151\ndash 184},
         url={http://dx.doi.org/10.1007/s00222-006-0513-0},
      review={\MR{2242636}},
}

\bib{BargeMorel}{article}{
      author={Barge, J.},
      author={Morel, F.},
       title={Groupe de {C}how des cycles orient\'es et classe d'{E}uler des
  fibr\'es vectoriels},
        date={2000},
        ISSN={0764-4442},
     journal={C. R. Acad. Sci. Paris S\'er. I Math.},
      volume={330},
      number={4},
       pages={287\ndash 290},
  url={http://dx.doi.org.prx.library.gatech.edu/10.1016/S0764-4442(00)00158-0},
      review={\MR{1753295}},
}

\bib{BhatSrid-Euler_class_group}{article}{
      author={Bhatwadekar, S.~M.},
      author={Sridharan, R.},
       title={The {E}uler class group of a {N}oetherian ring},
        date={2000},
        ISSN={0010-437X},
     journal={Compositio Math.},
      volume={122},
      number={2},
       pages={183\ndash 222},
  url={http://dx.doi.org.prx.library.gatech.edu/10.1023/A:1001872132498},
      review={\MR{1775418}},
}

\bib{BhatSrid-Zero_cycles}{article}{
      author={Bhatwadekar, S.~M.},
      author={Sridharan, R.},
       title={Zero cycles and the {E}uler class groups of smooth real affine
  varieties},
        date={1999},
        ISSN={0020-9910},
     journal={Invent. Math.},
      volume={136},
      number={2},
       pages={287\ndash 322},
         url={http://dx.doi.org.prx.library.gatech.edu/10.1007/s002220050311},
      review={\MR{1688449}},
}

\bib{BottTu}{book}{
      author={Bott, R.},
      author={Tu, L.~W.},
       title={Differential forms in algebraic topology},
      series={Graduate Texts in Mathematics},
   publisher={Springer-Verlag, New York-Berlin},
        date={1982},
      volume={82},
        ISBN={0-387-90613-4},
      review={\MR{658304}},
}

\bib{BBMMOTraceA1degree}{unpublished}{
      author={Brazelton, T.},
      author={Burklund, R.},
      author={McKean, S.},
      author={Montoro, M.},
      author={Opie, M.},
       title={The trace of the local $\bold{A}^1$-degree},
        date={2019},
        note={To appear in Homology Homotopy Appl., available at
  \url{https://arxiv.org/abs/1912.04788}},
}

\bib{Conner_Perlis}{book}{
      author={Conner, P.~E.},
      author={Perlis, R.},
       title={A survey of trace forms of algebraic number fields},
      series={Series in Pure Mathematics},
   publisher={World Scientific Publishing Co., Singapore},
        date={1984},
      volume={2},
        ISBN={9971-966-04-2; 9971-966-05-0},
         url={http://dx.doi.org.prx.library.gatech.edu/10.1142/0066},
      review={\MR{761569}},
}

\bib{DJK}{unpublished}{
      author={D\'eglise, F.},
      author={Jin, F.},
      author={Khan, A.},
       title={Fundamental classes in motivic homotopy theory},
        date={2018},
        note={{\em Preprint}, available at
  \url{https://arxiv.org/abs/1805.05920}},
}

\bib{EremenkoGabrielov-RationalFunctionsRealCriticalPoints}{article}{
      author={Eremenko, A.},
      author={Gabrielov, A.},
       title={Rational functions with real critical points and the {B}. and
  {M}. {S}hapiro conjecture in real enumerative geometry},
        date={2002},
        ISSN={0003-486X},
     journal={Ann. of Math. (2)},
      volume={155},
      number={1},
       pages={105\ndash 129},
         url={https://doi.org/10.2307/3062151},
      review={\MR{1888795}},
}

\bib{EisenbudHarris}{book}{
      author={Eisenbud, D.},
      author={Harris, J.},
       title={3264 and all that---a second course in algebraic geometry},
   publisher={Cambridge University Press, Cambridge},
        date={2016},
        ISBN={978-1-107-60272-4; 978-1-107-01708-5},
         url={https://doi-org.prx.library.gatech.edu/10.1017/CBO9781139062046},
      review={\MR{3617981}},
}

\bib{eisenbud77}{article}{
      author={Eisenbud, D.},
      author={Levine, H.~I.},
       title={An algebraic formula for the degree of a {$C^{\infty }$} map
  germ},
        date={1977},
     journal={Ann. of Math. (2)},
      volume={106},
      number={1},
       pages={19\ndash 44},
        note={With an appendix by Bernard Teissier, ``Sur une in{\'e}galit{\'e}
  {\`a} la Minkowski pour les multiplicit{\'e}s''},
      review={\MR{0467800 (57 \#7651)}},
}

\bib{FaselGroupesCW}{article}{
      author={Fasel, J.},
       title={Groupes de {C}how-{W}itt},
        date={2008},
        ISSN={0249-633X},
     journal={M\'em. Soc. Math. Fr. (N.S.)},
      number={113},
       pages={viii+197},
      review={\MR{2542148}},
}

\bib{Grig_Ivan}{article}{
      author={Grigor`ev, D.~Ju.},
      author={Ivanov, N.~V.},
       title={On the {E}isenbud-{L}evine formula over a perfect field},
        date={1980},
        ISSN={0002-3264},
     journal={Dokl. Akad. Nauk SSSR},
      volume={252},
      number={1},
       pages={24\ndash 27},
      review={\MR{572114}},
}

\bib{HWXZ-real_cycle}{unpublished}{
      author={Hornbostel, J.},
      author={Wendt, M.},
      author={Xie, H.},
      author={Zibrowius, M.},
       title={The Real Cycle Class Map},
        date={2019},
        note={{\em Preprint}, available at
  \url{https://arxiv.org/abs/1911.04150v2}},
}

\bib{Hoyois_lef}{article}{
      author={Hoyois, M.},
       title={A quadratic refinement of the
  {G}rothendieck-{L}efschetz-{V}erdier trace formula},
        date={2014},
        ISSN={1472-2747},
     journal={Algebr. Geom. Topol.},
      volume={14},
      number={6},
       pages={3603\ndash 3658},
  url={http://dx.doi.org.prx.library.gatech.edu/10.2140/agt.2014.14.3603},
      review={\MR{3302973}},
}

\bib{Wendt-Hornbostel}{article}{
      author={Hornbostel, J.},
      author={Wendt, M.},
       title={{C}how-{W}itt rings of classifying spaces of symplectic and
  special linear groups},
	date={2019},
	journal={J. Topol.},
	volume={12},
	number={3},
	pages={915\ndash965},
}

\bib{khimshiashvili}{article}{
      author={Khimshiashvili, G.~N.},
       title={The local degree of a smooth mapping},
        date={1977},
     journal={Sakharth. SSR Mecn. Akad. Moambe},
      volume={85},
      number={2},
       pages={309\ndash 312},
      review={\MR{0458467 (56 \#16670)}},
}

\bib{KWA1degree}{article}{
   author={Kass, J.},
   author={Wickelgren, K.},
   title={The class of Eisenbud-Khimshiashvili-Levine is the local
   $\bold{A}^1$-Brouwer degree},
   journal={Duke Math. J.},
   volume={168},
   date={2019},
   number={3},
   pages={429--469},
   issn={0012-7094},
   review={\MR{3909901}},
   doi={10.1215/00127094-2018-0046},
}

\bib{KWA1degreeClassical}{article}{
      author={Kass, J.},
      author={Wickelgren, K.},
       title={A classical proof that the algebraic homotopy class of a rational
  function is the residue pairing},
        date={2020},
         journal={Linear Algebra Appl.},
      volume={595},
       pages={157\ndash 181},
         url={https://arxiv.org/abs/1602.08129},
	review={\MR{4073493}},
}

\bib{CubicSurface}{unpublished}{
      author={Kass, J.},
      author={Wickelgren, K.},
       title={An arithmetic count of the lines on a smooth cubic surface},
        date={2017},
        note={To appear in Compos. Math., available at
  \url{https://arxiv.org/abs/1708.01175}},
}

\bib{KatoBilforms}{article}{
      author={Kato, K.},
       title={Symmetric bilinear forms, quadratic forms and {M}ilnor
              {$K$}-theory in characteristic two},
        date={1982},
        ISSN={0020-9910},
     journal={Invent. Math.},
      volume={66},
      number={3},
       pages={493\ndash 510},
         url={https://doi.org/10.1007/BF01389226},
      review={\MR{662605}},
}

\bib{lam05}{book}{
      author={Lam, T.~Y.},
       title={Introduction to quadratic forms over fields},
      series={Graduate Studies in Mathematics},
   publisher={American Mathematical Society, Providence, RI},
        date={2005},
      volume={67},
        ISBN={0-8218-1095-2},
      review={\MR{2104929}},
}

\bib{Levine-NormalCone}{unpublished}{
      author={Levine, M.},
       title={The intrinsic stable normal cone},
        date={2017},
        note={{\em Preprint}, available at
  \url{https://arxiv.org/abs/1703.03056}},
}

\bib{Levine-EC}{unpublished}{
      author={Levine, M.},
       title={Toward an enumerative geometry with quadratic forms},
        date={2017},
        note={{\em Preprint}, available at
  \url{https://arxiv.org/abs/1703.03049}},
}

\bib{Levine-Witt}{article}{
      author={Levine, M.},
       title={Motivic {E}uler characteristics and {W}itt-valued characteristic
  classes},
        date={2019},
	journal={Nagoya Math. J.},
	volume={236},
	pages={251 \ndash 310}, 
	url={https://doi.org/10.1017/nmj.2019.6},
	review={\MR{4094419}},
}

\bib{Levine-Welschinger}{unpublished}{
      author={Levine, M.},
       title={Toward an algebraic theory of {W}elschinger invariants},
        date={2018},
        note={{\em Preprint}, available at
  \url{https://arxiv.org/abs/1808.02238}},
}

\bib{LevineRaksit_MotivicGaussBonnet}{article}{
      author={Levine, M.},
      author={Raksit, A.},
       title={Motivic {G}auss--{B}onnet formulas},
	journal={Algebra Number Theory},
   	volume={14},
   	date={2020},
   	number={7},
   	pages={1801--1851},
   	doi={10.2140/ant.2020.14.1801},       
}

\bib{McKean-Bezout}{unpublished}{
      author={McKean, S.},
       title={An Arithmetic Enrichment of B\'ezout's Theorem},
        date={2020},
        note={{\em Preprint}, available at
  \url{https://arxiv.org/abs/2003.07413}},
}

\bib{Milnor_AlgK-theory_quadratic_forms}{article}{
      author={Milnor, J.},
       title={Algebraic {$K$}-theory and quadratic forms},
        date={1969/1970},
        ISSN={0020-9910},
     journal={Invent. Math.},
      volume={9},
       pages={318\ndash 344},
         url={https://doi.org/10.1007/BF01425486},
      review={\MR{0260844}},
}

\bib{morel}{book}{
      author={Morel, F.},
       title={{$\Bbb A^1$}-algebraic topology over a field},
      series={Lecture Notes in Mathematics},
   publisher={Springer, Heidelberg},
        date={2012},
      volume={2052},
        ISBN={978-3-642-29513-3},
         url={http://dx.doi.org/10.1007/978-3-642-29514-0},
      review={\MR{2934577}},
}

\bib{Mandal-Srid-Euler_classes_complete_intersections}{article}{
      author={Mandal, S.},
      author={Sridharan, R.},
       title={Euler classes and complete intersections},
        date={1996},
        ISSN={0023-608X},
     journal={J. Math. Kyoto Univ.},
      volume={36},
      number={3},
       pages={453\ndash 470},
         url={http://dx.doi.org.prx.library.gatech.edu/10.1215/kjm/1250518503},
      review={\MR{1417820}},
}

\bib{MukhinTarasovVarchenko-SchubertCalculusGL}{article}{
      author={Mukhin, E.},
      author={Tarasov, V.},
      author={Varchenko, A.},
       title={Schubert calculus and representations of the general linear
  group},
        date={2009},
        ISSN={0894-0347},
     journal={J. Amer. Math. Soc.},
      volume={22},
      number={4},
       pages={909\ndash 940},
         url={https://doi.org/10.1090/S0894-0347-09-00640-7},
      review={\MR{2525775}},
}

\bib{MukhinTarasovVarchenko-ShapiroConjecture}{article}{
      author={Mukhin, E.},
      author={Tarasov, V.},
      author={Varchenko, A.},
       title={The {B}. and {M}. {S}hapiro conjecture in real algebraic geometry
  and the {B}ethe ansatz},
        date={2009},
        ISSN={0003-486X},
     journal={Ann. of Math. (2)},
      volume={170},
      number={2},
       pages={863\ndash 881},
         url={https://doi.org/10.4007/annals.2009.170.863},
      review={\MR{2552110}},
}

\bib{morelvoevodsky1998}{article}{
      author={Morel, F.},
      author={Voevodsky, V.},
       title={{${\mathbb A}^1$}-homotopy theory of schemes},
        date={1999},
        ISSN={0073-8301},
     journal={Inst. Hautes \'Etudes Sci. Publ. Math.},
      number={90},
       pages={45\ndash 143 (2001)},
         url={http://www.numdam.org/item?id=PMIHES_1999__90__45_0},
      review={\MR{1813224 (2002f:14029)}},
}

\bib{okonek14}{article}{
      author={Okonek, C.},
      author={Teleman, A.},
      title={Intrinsic signs and lower bounds in
  real algebraic geometry}, 
      date={2014},
      journal={J. Reine Angew. Math.}
      volume={688},
      pages={219\ndash 241},
      review={\MR{3176620}},
 }
 
 \bib{OrlovVishikVoevodskay}{article}{
      author={Orlov, D.},
      author={Vishik, A.},
      author={Voevodsky, V.},
      title={An exact sequence for {$K^M_\ast/2$} with applications to
              quadratic forms}, 
      journal={Ann. of Math. (2)},
   volume={165},
   date={2007},
   number={1},
      pages={1\ndash 13},
      review={\MR{2276765}},
 }

\bib{Palamodov}{article}{
      author={Palamodov, V.~P.},
       title={The multiplicity of a holomorphic transformation},
        date={1967},
        ISSN={0374-1990},
     journal={Funkcional. Anal. i Prilo\v zen},
      volume={1},
      number={3},
       pages={54\ndash 65},
      review={\MR{0236424 (38 \#4720)}},
}

\bib{Pauli-M2}{unpublished}{
      author={Pauli, S.},
       title={Computing $\bold{A}^1$-Euler numbers with Macaulay2},
        date={2020},
        note={{\em Preprint}, available at
  \url{https://arxiv.org/abs/2003.01775}},
}

\bib{Pauli-quintic3fold}{unpublished}{
      author={Pauli, S.},
       title={Quadratic Types and dynamic Euler Numbers of Lines on a QuinticThreefold},
        date={2020},
        note={{\em Thesis}},
}

\bib{Shafarevich_Basic_algebraic_geometry1}{book}{
      author={Shafarevich, I.~R.},
       title={Basic algebraic geometry. 1},
     edition={Second},
   publisher={Springer-Verlag, Berlin},
        date={1994},
        ISBN={3-540-54812-2},
        note={Varieties in projective space, Translated from the 1988 Russian
  edition and with notes by Miles Reid},
      review={\MR{1328833}},
}

\bib{Sottile-Real_solutions_to_equations_from_geometry}{book}{
      author={Sottile, F.},
       title={Real solutions to equations from geometry},
      series={University Lecture Series},
   publisher={American Mathematical Society, Providence, RI},
        date={2011},
      volume={57},
        ISBN={978-0-8218-5331-3},
         url={https://doi.org/10.1090/ulect/057},
      review={\MR{2830310}},
}

\bib{Sottile-enumerative_geometry_for_real_Grassmannian_lines_Pn}{article}{
      author={Sottile, F.},
       title={Enumerative geometry for the real {G}rassmannian of lines in
  projective space},
        date={1997},
        ISSN={0012-7094},
     journal={Duke Math. J.},
      volume={87},
      number={1},
       pages={59\ndash 85},
         url={https://doi.org/10.1215/S0012-7094-97-08703-2},
      review={\MR{1440063}},
}

\bib{scheja}{article}{
      author={Scheja, G.},
      author={Storch, U.},
       title={\"{U}ber {S}purfunktionen bei vollst\"andigen {D}urchschnitten},
        date={1975},
        ISSN={0075-4102},
     journal={J. Reine Angew. Math.},
      volume={278/279},
       pages={174\ndash 190},
      review={\MR{0393056 (52 \#13867)}},
}

\bib{Vakil_Schubert_induction}{article}{
      author={Vakil, R.},
       title={Schubert induction},
        date={2006},
        ISSN={0003-486X},
     journal={Ann. of Math. (2)},
      volume={164},
      number={2},
       pages={489\ndash 512},
         url={https://doi.org/10.4007/annals.2006.164.489},
      review={\MR{2247966}},
}

\bib{Voevodsky-Motivic_cohomology_Z2-coeffs}{article}{
      author={Voevodsky, V.},
       title={Motivic cohomology with {${\bf Z}/2$}-coefficients},
        date={2003},
        ISSN={0073-8301},
     journal={Publ. Math. Inst. Hautes \'{E}tudes Sci.},
      number={98},
       pages={59\ndash 104},
         url={https://doi.org/10.1007/s10240-003-0010-6},
      review={\MR{2031199}},
}

\bib{Voevodsky-Reduced_power_ops_motivic_cohomology}{article}{
      author={Voevodsky, V.},
       title={Reduced power operations in motivic cohomology},
        date={2003},
        ISSN={0073-8301},
     journal={Publ. Math. Inst. Hautes \'{E}tudes Sci.},
      number={98},
       pages={1\ndash 57},
         url={https://doi.org/10.1007/s10240-003-0009-z},
      review={\MR{2031198}},
}

\bib{Wendt-Chow-Witt_Gr}{unpublished}{
      author={Wendt, M.},
       title={{C}how-{W}itt rings of {G}rassmannians},
        date={2018},
        note={{\em Preprint} available at
  \url{https://arxiv.org/abs/1805.06142}},
}

\bib{Wendt-oriented_schubert}{article}{
      author={Wendt, M.},
       title={Oriented Schubert calculus in Chow--Witt rings of Grassmannians},
        date={2020},
     journal={Contemp. Math. Motivic homotopy theory and refined enumerative geometry},
      publisher={Amer. Math. Soc., Providence, RI},
      volume={745},
       pages={217\ndash 267},
      review={\MR{4071217}},
}


\end{biblist}
\end{bibdiv}

\end{document}